\documentclass[10pt]{article}
        
\usepackage{amssymb,amsmath,amsfonts,amsthm}
\usepackage[a4paper,includeheadfoot,left=4cm,right=4cm,top=3cm,bottom=3cm]{geometry}

\usepackage[numbers,comma,sort&compress]{natbib}
\usepackage[stretch=10]{microtype} 

\makeatletter
\let\@fnsymbol\@arabic
\makeatother

\usepackage{color}
\definecolor{marin}{rgb}{0.,0.3,0.7}
\usepackage[colorlinks,citecolor=marin,linkcolor=marin,urlcolor=marin,bookmarksopen,bookmarksnumbered,breaklinks=true]{hyperref}

\newcommand{\ep}{\varepsilon}

\newcommand{\Z}{\mathbb{Z}}
\newcommand{\N}{\mathbb{N}}
\newcommand{\C}{\mathbb{C}}
\newcommand{\dbf}{\mathbf{d}}

\newcommand{\ubf}{\mathbf{u}}
\newcommand{\vbf}{\mathbf{v}}
\newcommand{\zbf}{\mathbf{z}}
\newcommand{\Abf}{\mathbf{A}}
\newcommand{\Bbf}{\mathbf{B}}
\newcommand{\Kbf}{\mathbf{K}}
\newcommand{\Labf}{\boldsymbol{\Lambda}}
\newcommand{\Phbf}{\boldsymbol{\Phi}}
\providecommand{\abs}[1]{\lvert#1\rvert}
\providecommand{\absbig}[1]{\bigl\lvert#1\bigr\rvert}
\providecommand{\absbigg}[1]{\biggl\lvert#1\biggr\rvert}
\providecommand{\absBig}[1]{\Bigl\lvert#1\Bigr\rvert}
\providecommand{\kla}[1]{(#1)}
\providecommand{\klabig}[1]{\bigl(#1\bigr)}
\providecommand{\klabigg}[1]{\biggl(#1\biggr)}
\providecommand{\klaBig}[1]{\Bigl(#1\Bigr)}
\providecommand{\norm}[1]{\lVert#1\rVert}
\providecommand{\normbig}[1]{\bigl\lVert#1\bigr\rVert}
\providecommand{\normbigg}[1]{\biggl\lVert#1\biggr\rVert}

\providecommand{\normv}[1]{\ensuremath{{\lvert\hskip-1pt\lvert\hskip-1pt\lvert}#1{\rvert\hskip-1pt\rvert\hskip-1pt\rvert}}}
\providecommand{\normvbig}[1]{\ensuremath{{\bigl\lvert\hskip-1pt\bigl\lvert\hskip-1pt\bigl\lvert}#1{\bigr\rvert\hskip-1pt\bigr\rvert\hskip-1pt\bigr\rvert}}}
\providecommand{\normvbigg}[1]{\ensuremath{{\biggl\lvert\hskip-1pt\biggl\lvert\hskip-1pt\biggl\lvert}#1{\biggr\rvert\hskip-1pt\biggr\rvert\hskip-1pt\biggr\rvert}}}
\providecommand{\normvBig}[1]{\ensuremath{{\Bigl\lvert\hskip-1pt\Bigl\lvert\hskip-1pt\Bigl\lvert}#1{\Bigr\rvert\hskip-1pt\Bigr\rvert\hskip-1pt\Bigr\rvert}}}
\providecommand{\skla}[1]{\langle#1\rangle}
\newcommand{\formulatext}[1]{\qquad\text{#1}\qquad}
\newcommand\myfor{\formulatext{for}}

\newcommand\myand{\formulatext{and}}

\newcommand{\iu}{\mathrm{i}}
\newcommand{\e}{\mathrm{e}}
\newcommand{\drm}{\mathrm{d}}
\DeclareMathOperator{\ReT}{Re}
\DeclareMathOperator{\sinc}{sinc}
\numberwithin{equation}{section}
\newtheorem{mytheorem}{Theorem}[section]
\newtheorem{mylemma}[mytheorem]{Lemma}

\theoremstyle{definition}
\newtheorem{algorithm}[mytheorem]{Algorithm}
\newtheorem{mydefinition}[mytheorem]{Definition}
\newtheorem*{definition*}{Definition}
\newtheorem*{remark*}{Remark}
\newtheorem*{approximationansatz}{Approximation ansatz}
\newtheorem*{approximationansatz2}{Approximation ansatz, part 2}
\newtheorem*{system}{System for the modulation functions}
\newtheorem*{system2}{System for the modulation coefficients functions}
\newcommand{\disc}{\mathcal{K}}
\newcommand{\Ec}{\mathcal{E}}
\newcommand{\np}{{\hat{n}}}
\newcommand{\sfrac}[2]{\mbox{\footnotesize$\displaystyle\frac{#1}{#2}$}}
\newcommand{\nt}{\alpha}
\usepackage[bbgreekl]{mathbbol}
\DeclareSymbolFontAlphabet{\mathbb}{AMSb}
\DeclareSymbolFontAlphabet{\mathbbl}{bbold}
\newcommand{\zhat}{\mathbbl{z}}
\newcommand{\vhat}{\mathbbl{v}}
\newcommand{\ghat}{\mathbbl{g}}
\newcommand{\Ahat}{\mathbbl{A}}
\newcommand{\Bhat}{\mathbbl{B}}
\newcommand{\Fhat}{\mathbbl{F}}
\newcommand{\Phihat}{\mathbbl{\Phi}}
\newcommand{\idhat}{\mathbbl{1}}
\newcommand{\zind}{\zhat}
\newcommand{\vind}{\vhat}

\title{High-order splitting integrators for nonlinear Schr\"odinger equations over long times}

\author{Ludwig Gauckler\,\thanks{Institut f\"ur Mathematik,
          Freie Universit\"at Berlin,
          Arnimallee 9,
          D-14195 Berlin, Germany
          ({\tt gauckler@math.fu-berlin.de}).}
}

\date{Version of 28 February 2018}

\begin{document}

\maketitle

\begin{abstract}
The long-time behaviour of splitting integrators applied to nonlinear Schr\"o\-din\-ger equations in a weakly nonlinear setting is studied. It is proven that the energy is nearly conserved on long time intervals. The analysis includes all consistent splitting integrators with real-valued coefficients, in particular splitting integrators of high order. The proof is based on a completely resonant modulated Fourier expansion in time of the numerical solution. \\[1.5ex]
\end{abstract}

\section{Introduction}

The long-time behaviour of symplectic integrators applied to (non-oscillatory) Ha\-mil\-to\-ni\-an ordinary differential equations is by now well understood. Their long-time behaviour in the case of Hamiltonian partial differential equations is, however, less clear. One of the open questions is the behaviour of higher order integrators for such equations, and this is the topic of the present paper which studies high-order splitting integrators \cite{Blanes2016,Hairer2006,McLachlan2002}.

As a model problem of a Hamiltonian partial differential equation, we consider the cubic nonlinear Schr\"odinger equation 
\begin{equation}\label{eq-nls}
\iu \partial_t \psi = - \partial_{x}^2 \psi + \abs{\psi}^2 \psi, \qquad \psi=\psi(x,t).
\end{equation}
We consider this equation in a weakly nonlinear regime of small initial values~$\psi(\cdot,0)$ at time~$t=0$. In space, we impose $2\pi$-periodic boundary conditions in one dimension, i.e., $x\in\mathbb{T}=\mathbb{R}/(2\pi\mathbb{Z})$. The conserved energy of this Hamiltonian partial differential equation is
\[
E\klabig{\psi(\cdot,t)} = \frac{1}{2\pi} \int_{\mathbb{T}} \abs{\partial_x \psi(x,t)}^2\,\drm x + \frac1{4\pi} \int_{\mathbb{T}} \abs{\psi(x,t)}^4\,\drm x.
\]

For this equation, long-time near-conservation of energy and other long-time properties have been studied and proven for the first-order Lie--Trotter and second-order Strang splitting integrators in \cite{Bambusi2013,Faou2012,Faou2014,Faou2011,Faou2009a,Faou2009b,Gauckler2017a,Gauckler2010a,Gauckler2010b} (when combined with a Fourier collocation in space).
In the present paper, we prove long-time near-conservation of energy for arbitrary splitting integrators, in particular splitting integrators of order higher than two. Such higher order splitting integrators are successfully applied to nonlinear Schr\"odinger equations and have been analyzed on bounded time intervals, see, for example, \cite{Bao2005,Koch2013a,Thalhammer2012a,Thalhammer2009}, but their long-time behaviour in the case of nonlinear Schr\"odinger equations (or Hamiltonian differential equations in general) has not been studied so far. 

We finally mention a different line of research \cite{Barletti2018,Brugnano2016,Celledoni2012}, where energy-preserving (but not necessarily symplectic and explicit) methods are constructed, used and analyzed.

The paper is organized as follows. In Section~\ref{sec-method} we introduce the considered splitting integrators of arbitrary order. The main result on the long-time behaviour of these integrators is stated in the following Section~\ref{sec-main}. The proof of this result is given in Sections~\ref{sec-mfe}--\ref{sec-long} using modulated Fourier expansions in time. We conclude in Section~\ref{sec-conclusion} by pointing out extensions of the developed theory.

\section{Splitting integrators}\label{sec-method}

For the numerical discretization of the nonlinear Schr\"odinger equation~\eqref{eq-nls}, we consider splitting integrators in time applied to a pseudo-spectral (Fourier) discretization in space. In particular and in contrast to previous work on the long-time behaviour of numerical methods, we allow for higher order splitting integrators.

We denote by~$h$ the time step-size and by~$t_n=n h$ discrete times. The considered fully discrete methods compute approximations $\psi^n$ to the exact solution at time~$t_n$ which are trigonometric polynomials of degree~$K$:
\[
\psi(x,t_n) \approx \psi^n(x) = \sum_{j\in\disc} \psi_j^n \, \e^{\iu j x}, \qquad \disc = \bigl\{-K,-K+1,\dots,K-1\bigr\}. 
\]
To describe the time stepping from~$\psi^n$ to~$\psi^{n+1}$ with a splitting integrator, we introduce the time-$\nt$ flows~$\phi^\nt_{A}$ and~$\phi^\nt_{B}$ of the differential equations
\[
\iu \partial_t \psi = A \psi \myand \iu \partial_t \psi = B (\psi)
\]
with the linear operator $A\psi = -\partial_x^2\psi$ and the nonlinear operator $B(\psi) = \mathcal{Q}^K \klabig{\abs{\psi}^2 \psi}$, where~$\mathcal{Q}^K$ denotes the trigonometric interpolation by trigonometric polynomials of degree~$K$. 
These two differential equations split the nonlinear Schr\"odinger equation, or more precisely its pseudo-spectral Fourier discretization in space, into two parts. 
With the corresponding flows~$\phi^\nt_{A}$ and~$\phi^\nt_{B}$ a splitting integrator reads
\begin{equation}\label{eq-method}
  \psi^{n+1} = \phi^{a_1h}_{A} \circ \phi^{b_1h}_{B} \circ \phi^{a_2h}_{A} \circ \phi^{b_2h}_{B} \circ \dotsm \circ \phi^{a_sh}_{A} \circ \phi^{b_sh}_{B} (\psi^n)
\end{equation}
with coefficients $a_1,\dots,a_s$ and $b_1,\dots,b_s$.
The initial value $\psi(\cdot,0)$ of the nonlinear Schr\"odinger equation~\eqref{eq-nls} at time $t_0=0$ is approximated by
$\psi^0 = \mathcal{Q}^K \kla{\psi(\cdot,0)}$.
The computational efficiency of such splitting methods comes from the fact that the flows can be computed exactly via
\begin{equation}\label{eq-flows}
  \phi^{\nt}_{A} (\eta) = \sum_{j\in\disc} \e^{-\iu j^2 \nt} \, \eta_j \, \e^{\iu j x} \myand \phi^{\nt}_{B} (\eta) = \mathcal{Q}^K \klabig{\e^{-\iu \nt \abs{\eta}^2} \eta }
\end{equation}
for $\eta(x) = \sum_{j\in\disc} \eta_j \e^{\iu j x}$.

An interesting question is of course how to choose the coefficients $a_1,\dots,a_s$ and $b_1,\dots,b_s$ in the method~\eqref{eq-method}. The standard choices are the (first-order) Lie--Trotter splittings 
(first variant: $s=1$, $a_1 = b_1 = 1$; second variant: $s = 2$, $a_1=b_2=0$, $a_2=b_1=1$)
and the (symmetric and second-order) Strang splittings 
(first variant: $s=2$, $a_1=a_2=1/2$, $b_1=1$, $b_2=0$; second variant: $s=2$, $a_1=0$, $a_2=1$, $b_1=b_2=1/2$).
It turns out, however, that also methods~\eqref{eq-method} with order higher than two are successful.
Classical examples are the methods of Suzuki~\cite{Suzuki1990} and Yoshida~\cite{Yoshida1990}, and particularly successful methods have been derived for example by Blanes~\& Moan~\cite{Blanes2002}.
For a general introduction to (higher order) splitting integrators we refer to \cite[Chapter~3]{Blanes2016}, \cite[Chapters~II and III]{Hairer2006} and \cite{McLachlan2002}; for many competitive examples of such methods we refer to \cite[Section~3.7]{Blanes2016}, \cite[Table~1]{Thalhammer2009} and the references therein; and for their use in the case of Schr\"odinger equations we refer to \cite{Bao2005,Koch2013a,Thalhammer2012a,Thalhammer2009}, for example.

All the splitting integrators mentioned above are covered by the analysis to be presented below. We only assume throughout the paper that the coefficients $a_1,\dots,a_s$ and $b_1,\dots,b_s$ are real-valued and that the method is consistent, i.e.,
\begin{equation}\label{eq-consistency}
a_1 + a_2 + \dots + a_s = 1 \myand b_1 + b_2 + \dots + b_s = 1.
\end{equation}

\section{Statement of the main result}\label{sec-main}

We denote by~$\norm{\eta}_1$ the Sobolev $H^1(\mathbb{T})$-norm of $\eta(x) = \sum_{j\in\Z} \eta_j \e^{\iu j x}$:
\begin{equation}\label{eq-norm}
\norm{\eta}_1^2 = \frac{1}{2\pi} \int_{\mathbb{T}} \klabig{ \abs{\eta(x)}^2 + \abs{\partial_x\eta(x)}^2 } \,\drm x = \sum_{j\in\mathbb{Z}} \skla{j}^{2} \abs{\eta_j}^2
\end{equation}
with the weights
\begin{equation}\label{eq-skla}
\skla{j} = \sqrt{\abs{j}^2+1}.
\end{equation}
Our main result is the following theorem on the long-time near-conservation of this norm by the numerical method~\eqref{eq-method}. The proof of this theorem will be given in Sections~\ref{sec-mfe}--\ref{sec-long} below. 

\begin{mytheorem}\label{thm-main}
Let~$N\ge 2$ and assume that the CFL-type step-size restriction
\begin{equation}\label{eq-cfl}
(N+1) h K^2 \le c_0 < 2\pi
\end{equation}
on the discretization parameters~$h$ and~$K$ holds. For a small initial value 
\begin{equation}\label{eq-init}
\norm{\psi^0}_1 \le \ep
\end{equation}
with $\ep\le \ep_0$ the splitting integrator~\eqref{eq-method} then nearly preserves the squared $H^1(\mathbb{T})$-norm on long time intervals:
\[
\ep^{-2} \absbig{\norm{\psi^n}_1^2 - \norm{\psi^0}_1^2} \le C \ep \myfor 0\le t_n=n h \le \ep^{-N+1}.
\]
The constants~$C$ and~$\ep_0$ depend on~$N$ and~$c_0$ from~\eqref{eq-cfl}, on the number of stages~$s$ and on an upper bound on the absolute values of the coefficients $b_1,\dots,b_s$ of the method~\eqref{eq-method}, but they are independent of the size~$\ep$ of the initial value, the time step-size~$h$ and the spatial discretization parameter~$K$.
\end{mytheorem}

As shown in \cite[Corollary~3.2]{Gauckler2017a}, the long-time near-conservation of the squared $H^1(\mathbb{T})$-norm of Theorem~\ref{thm-main} implies in particular 
long-time near-conservation of energy
\[
\ep^{-2} \absbig{E(\psi^n) - E(\psi^0) } \le C \ep \myfor 0\le t_n=n h \le \ep^{-N+1}
\]
along the numerical solution~\eqref{eq-method} and its long-time regularity
\[
\norm{\psi^n}_1 \le 2\ep \myfor 0\le t_n=n h \le \ep^{-N+1}.
\]

The main assumption on the numerical method in Theorem~\ref{thm-main} is the CFL-type step-size restriction~\eqref{eq-cfl}. With additional technical effort, this condition could be improved to $(N+1) h K^2 \le c_0<4\pi$ and even weakened to a numerical non-resonance condition as in \cite[Theorem~3.1]{Gauckler2017a}.

An important assumption on the problem in Theorem~\ref{thm-main} is the smallness~\eqref{eq-init} of the initial values, which cannot be removed at present. In contrast, there is some flexibility in choosing the nonlinearity in the nonlinear Schr\"odinger equation~\eqref{eq-nls}. For example, the cubic nonlinearity could be replaced (with minor modifications) by $\kappa \abs{\psi}^{2p}\psi$ with $\kappa\in\mathbb{R}$ and $p\in\mathbb{N}$.

Theorem~\ref{thm-main} shows that high-order splitting integrators have (at least) the same long-time behaviour as has been proven for the (low-order) Lie--Trotter and Strang splitting integrators in~\cite{Gauckler2017a}. One open question is whether high-order methods behave better than low-order methods. Motivated by the situation for Hamiltonian ordinary differential equations and based on several numerical experiments (with an implementation in quadruple precision), we conjecture that the error in the energy should scale like~$h^p$ for a method of order~$p$ (in addition to being small in terms of~$\ep$). Our proof given here, however, does yield such a statement because it is based on expansions in terms of~$\ep$ in an essential way.

\section{Construction of a modulated Fourier expansion}\label{sec-mfe}

To begin with the proof of Theorem~\ref{thm-main}, we construct in this section a multiscale expansion of the numerical solution~\eqref{eq-method}, a modulated Fourier expansion in time. Such an expansion is one of the tools that can be used to prove long-time properties of numerical methods for Hamiltonian partial differential equations.

Throughout, we work under the assumptions of Theorem~\ref{thm-main}: We let~$N\ge 2$, we assume that the CFL-type step-size restriction~\eqref{eq-cfl} holds with this~$N$ and a constant $c_0< 2\pi$, and we consider a small initial value~\eqref{eq-init} of size~$\ep$. 

\subsection{The numerical method in Fourier space}

Throughout this section, we will consider the numerical method~\eqref{eq-method} in Fourier space. To do so we write the flows~$\phi_A^{\nt}$ and~$\phi_B^{\nt}$ given by~\eqref{eq-flows} in Fourier space. Denoting by an index~$j$ the $j$th Fourier coefficient, we have
\begin{equation}\label{eq-flowA-fourier}
\klabig{ \phi_A^{\nt}(\eta) }_j = \e^{-\iu j^2 \nt} \eta_j, 
\end{equation}
and 
\begin{equation}\label{eq-flowB-fourier}
\klabig{ \phi_B^{\nt}(\eta) }_j = \sum_{m=0}^{\infty} \frac{(-\iu \nt)^m}{m!} \!\!\!\! \sum_{\substack{j_1+\dots+j_{m+1}\\-j_{m+2}-\dots-j_{2m+1} \equiv j}} \!\!\!\! \eta_{j_1} \dotsm \eta_{j_{m+1}} \overline{\eta_{j_{m+2}} \dotsm \eta_{j_{2m+1}} }.
\end{equation}
Here and in the following, we denote by~$\equiv$ the congruence modulo~$2K$.

\subsection{Functional analytic framework}

\begin{mydefinition}\label{definition-zbf}
  In the following, we will consider sequences from the space $\mathbb{C}^{\disc\times\mathbb{Z}}$ indexed by $j\in\disc=\{-K,-K+1,\dots,K-1\}$ and $k\in\Z$. They are denoted by boldface letters:
\[
  \zbf=\klabig{z_j^k}_{j\in\disc,k\in\mathbb{Z}} \in \mathbb{C}^{\disc\times\mathbb{Z}}.
\]
For such a sequence, we introduce, for $0\le\sigma\le 1$, the norms 
\begin{equation}\label{eq-normv}
\normv{\zbf}_{\sigma}^2 = \sum_{j\in\disc} \skla{j}^{2\sigma} \klabigg{ \sum_{k\in\Z} \absbig{z_{j}^k} }^2
\end{equation}
with the weights~$\skla{j}=\sqrt{\abs{j}^2+1}$ of~\eqref{eq-skla}.
We also introduce the notations
\begin{equation}\label{eq-zbar}
  \overline{\zbf}=\klabig{\overline{z_{-j}^{-k}}}_{j\in\disc,k\in\mathbb{Z}} \in \mathbb{C}^{\disc\times\mathbb{Z}} \myand \abs{\zbf} = \klabig{\absbig{z_{j}^{k}}}_{j\in\disc,k\in\mathbb{Z}} \in \mathbb{C}^{\disc\times\mathbb{Z}}.
\end{equation}
In the definition of~$\overline{\zbf}$, we identify the index $-j=+K\notin\disc$ with the index $-K\in\disc$.
\end{mydefinition}

The norm~$\normv{\cdot}_\sigma$ of~\eqref{eq-normv} is the norm considered also in \cite[Section~5.4]{Gauckler2017a}. We recall from there that
\begin{equation}\label{eq-l1l2}
\sum_{j\in\disc} \skla{j}^{2\sigma} \sum_{k\in\Z} \absbig{z_{j}^k}^2 \le \normv{\zbf}_{\sigma}^2
\end{equation}
and that, for $\sigma>\frac12$, 
\begin{equation}\label{eq-algebra}
\normv{\zbf\ast\vbf}_{\sigma} \le C_\sigma \normv{\zbf}_\sigma \normv{\vbf}_\sigma,
\end{equation}
where~$\ast$ is the discrete convolution
\begin{equation}\label{eq-conv}
\kla{\zbf \ast \vbf }_j^k = \sum_{j_1+j_2\equiv j} \;\; \sum_{k_1+k_2=k} z_{j_1}^{k_1} v_{j_2}^{k_2}. 
\end{equation}
Recall that~$\equiv$ denotes the congruence modulo~$2K$.

\begin{mydefinition}\label{definition-zhat}
In addition, we will consider sequences from the space $\mathbb{C}^{\disc\times\mathbb{Z}\times\mathbb{N}}$ indexed by $j\in\disc$, $k\in\Z$ and $p\in\mathbb{N}$. These are denoted by blackboard letters:
\[
\zhat=\klabig{z_{j,p}^k}_{j\in\disc,k\in\mathbb{Z},p\in\mathbb{N}} \in \mathbb{C}^{\disc\times\mathbb{Z}\times\mathbb{N}}.
\]
By~$\idhat$, we denote the identity operator on this space~$\mathbb{C}^{\disc\times\mathbb{Z}\times\mathbb{N}}$.
For $p\in\N$ and a sequence $\zhat\in\mathbb{C}^{\disc\times\mathbb{Z}\times\mathbb{N}}$, we denote by~$\zind_p$ the sequence
\begin{equation}\label{eq-zhatp}
\zind_p = \klabig{z_{j,p}^k}_{j\in\disc,k\in\mathbb{Z}} \in \mathbb{C}^{\disc\times\mathbb{Z}}
\end{equation}
from the smaller space $\mathbb{C}^{\disc\times\mathbb{Z}}$. We finally introduce, for $0\le\sigma\le 1$ and $p\in\N$, the seminorms
\begin{equation}\label{eq-seminorm}
\abs{\zhat}_{\sigma,p}^2 = \normvbigg{\sum_{q=1}^p \abs{\zind_q} }_{\sigma}^2=\sum_{j\in\disc} \skla{j}^{2\sigma} \klabigg{ \sum_{k\in\Z} \sum_{q=1}^p \absbig{z_{j,q}^k} }^2,
\end{equation}
where the notations~\eqref{eq-zbar} and~\eqref{eq-zhatp} are used.
\end{mydefinition}

\subsection{Modulated Fourier expansion}

\begin{approximationansatz} We use a modulated Fourier expansion in time of the numerical solution~\eqref{eq-method}. This is an expansion of the form 
\begin{equation}\label{eq-mfe}
\widetilde{\psi}(x,t) = \sum_{j\in\disc} \sum_{k\in\Z} z_j^k\kla{\ep t} \e^{-\iu kt} \e^{\iu j x}
\end{equation}
such that
\begin{equation}\label{eq-mfe-approx}
\psi^n(x) \approx \widetilde{\psi}^n := \widetilde{\psi}(x,t_n) \myfor n=0,1,2,\ldots.
\end{equation}
The coefficients $z_j^k=z_j^k(\ep t)$ of the expansion are assumed to be polynomials of the slow time
\[
\tau = \ep t. 
\]
Recall that $\ep\ll 1$ is the small parameter of~\eqref{eq-init}.
The coefficients~$z_j^k$ are called \emph{modulation functions} and collected in the sequence $\zbf = \kla{z_j^k}_{j\in\disc,k\in\mathbb{Z}} \in \mathbb{C}^{\disc\times\mathbb{Z}}$ (see Definition~\ref{definition-zbf}).
\end{approximationansatz}

\begin{definition*} To derive equations for these modulation functions, we introduce operators~$\Phbf_\Abf^{\nt}$ and~$\Phbf_\Bbf^{\nt}$ on~$\mathbb{C}^{\disc\times\mathbb{Z}}$ such that, for $\vbf=(v_j^k)_{j\in\disc,k\in\Z}$ and $\eta(x) = \sum_{j\in\disc} \eta_j \e^{\iu j x}$ with $\eta_j=\sum_{k\in\Z} v_j^k \e^{-\iu k t}$, 
\begin{equation}\label{eq-prop-phiA}
\klabig{\phi_A^{\nt}(\eta)}_j = \sum_{k\in\Z} \klabig{ \Phbf_\Abf^{\nt}(\vbf) }_j^k \e^{-\iu kt} 
\end{equation}
and
\begin{equation}\label{eq-prop-phiB}
\klabig{\phi_B^{\nt}(\eta)}_j = \sum_{k\in\Z} \klabig{ \Phbf_\Bbf^{\nt}(\vbf) }_j^k \e^{-\iu kt} 
\end{equation}
with the flows~$\phi_A^{\nt}$ and~$\phi_B^{\nt}$ of~\eqref{eq-flowA-fourier} and~\eqref{eq-flowB-fourier}. 
The operator~$\Phbf_\Abf^{\nt}$ is given by
\begin{equation}\label{eq-Phbf-A}
\klabig{ \Phbf_\Abf^{\nt}(\vbf) }_j^k = \e^{-\iu j^2 \nt} v_j^k
\end{equation}
and the operator~$\Phbf_\Bbf^{\nt}$ by
\begin{equation}\label{eq-Phbf-B}
\Phbf_\Bbf^{\nt}(\vbf) = \sum_{m=0}^{\infty} \frac{(-\iu \nt)^m}{m!} \; \underbrace{\vbf \ast \dotsm \ast \vbf}_{\text{$m+1$ times}} \ast \underbrace{\overline{\vbf} \ast \dotsm \ast \overline{\vbf}}_{\text{$m$ times}},
\end{equation}
where~$\overline{\vbf}$ is defined in~\eqref{eq-zbar} and~$\ast$ in~\eqref{eq-conv}. The property~\eqref{eq-algebra} of the convolution~$\ast$ shows that $\Phbf_\Bbf^{\nt}(\vbf)$ has a finite norm~$\normv{\cdot}_1$ (see~\eqref{eq-normv}) if the argument~$\vbf$ does; in particular, the infinite sums in the definition of~$\Phbf_\Bbf^{\nt}$ converge in this case.
\end{definition*}

\begin{system} Using the properties~\eqref{eq-prop-phiA} and~\eqref{eq-prop-phiB}, we get
\begin{align*}
  &\klaBig{\phi^{a_1h}_{A} \circ \phi^{b_1h}_{B} \circ \phi^{a_2h}_{A} \circ \phi^{b_2h}_{B} \circ \dotsm \circ \phi^{a_sh}_{A} \circ \phi^{b_sh}_{B} (\eta)}_j \\
  & \qquad\qquad = \sum_{k\in\Z} \klaBig{ \Phbf^{a_1h}_{\Abf} \circ \Phbf^{b_1h}_{\Bbf} \circ \Phbf^{a_2h}_{\Abf} \circ \Phbf^{b_2h}_{\Bbf} \circ \dotsm \circ \Phbf^{a_sh}_{\Abf} \circ \Phbf^{b_sh}_{\Bbf} \kla{\vbf} }_j^k \e^{-\iu k t}.
\end{align*}
Inserting the ansatz~\eqref{eq-mfe} in the numerical method~\eqref{eq-method} and requiring~\eqref{eq-mfe-approx} thus yields
\[
\klabig{\widetilde\psi(\cdot,t+h)}_j = \sum_{k\in\Z} \klaBig{ \Phbf^{a_1h}_{\Abf} \circ \Phbf^{b_1h}_{\Bbf} \circ \Phbf^{a_2h}_{\Abf} \circ \Phbf^{b_2h}_{\Bbf} \circ \dotsm \circ \Phbf^{a_sh}_{\Abf} \circ \Phbf^{b_sh}_{\Bbf} \klabig{\zbf(\tau)} }_j^k \e^{-\iu k t}
\]
with $\tau=\ep t$.
By comparing the coefficients of $\e^{- \iu k t}$, we finally get the system
\begin{equation}\label{eq-modsystem}
\e^{-\iu k h} z_j^k(\tau+\ep h) = \klaBig{ \Phbf^{a_1h}_{\Abf} \circ \Phbf^{b_1h}_{\Bbf} \circ \Phbf^{a_2h}_{\Abf} \circ \Phbf^{b_2h}_{\Bbf} \circ \dotsm \circ \Phbf^{a_sh}_{\Abf} \circ \Phbf^{b_sh}_{\Bbf} \klabig{\zbf(\tau)} }_j^k
\end{equation}
that the modulation functions~$z_j^k$ of~\eqref{eq-mfe} and~\eqref{eq-mfe-approx} should satisfy, at least up to a small defect.
 For later use, we expand the polynomial on the left-hand side of~\eqref{eq-modsystem} as a Taylor series and thus get
\begin{equation}\label{eq-modsystem-temp}
\e^{-\iu k h} \sum_{\ell=0}^\infty \frac{\ep^\ell h^\ell}{\ell!} \,  \frac{\drm^\ell}{\drm\tau^\ell} \, z_j^k(\tau) = \klaBig{ \Phbf^{a_1h}_{\Abf} \circ \Phbf^{b_1h}_{\Bbf} \circ \dotsm \circ \Phbf^{a_sh}_{\Abf} \circ \Phbf^{b_sh}_{\Bbf} \klabig{\zbf(\tau)} }_j^k.
\end{equation}
For the initial value, we get the condition
\begin{equation}\label{eq-modsystem-init}
\sum_{k\in\mathbb{Z}} z_j^k(0) = \psi_j^0. 
\end{equation}
\end{system}

\begin{remark*}
  The modulated Fourier expansion~\eqref{eq-mfe} used here is in principle the same as in~\cite[Section~5.1]{Gauckler2017a}. As there and previously in~\cite{Gauckler2016}, it is a \emph{completely resonant modulated Fourier expansion} where the integer exponents~$k$ of $\e^{-\iu k t}$ are integer linear combinations of the completely resonant frequencies~$j^2$ of the nonlinear Schr\"odinger equation~\eqref{eq-nls} and the numerical method~\eqref{eq-method}. The system~\eqref{eq-modsystem} that the modulation functions have to satisfy, however, is more complicated than in~\cite[Section~5.1]{Gauckler2017a} involving several nonlinear operators instead of just one.
\end{remark*}

\subsection{Asymptotic expansion of the modulation functions}

\begin{approximationansatz2} To construct the modulation functions~$z_j^k$ of the modulated Fourier expansion~\eqref{eq-mfe}, we introduce yet another expansion. This time, it is a truncated expansion of~$z_j^k$ in terms of powers of the small parameter~$\ep$ of~\eqref{eq-init}:
\begin{equation}\label{eq-mfe-exp}
z_j^k(\tau) = \sum_{p=1}^N \ep^p z_{j,p}^k(\tau). 
\end{equation}
The truncation parameter~$N$ is the fixed but arbitrary number of Theorem~\ref{thm-main} and the step-size restriction~\eqref{eq-cfl}, and we set
\[
  z_{j,p}^k=0 \myfor p>N.
\]
The coefficients~$z_{j,p}^k$ are called \emph{modulation coefficient functions} and collected in the sequence $\zhat=\kla{z_{j,p}^k}_{j\in\disc,k\in\mathbb{Z},p\in\mathbb{N}} \in \mathbb{C}^{\disc\times\mathbb{Z}\times\mathbb{N}}$ (see Definition~\ref{definition-zhat}).
\end{approximationansatz2}

\begin{definition*}
To derive equations for these coefficients, we introduce operators $\Phihat_\Ahat^{\nt}$ and $\Phihat_\Bhat^{\nt}$ as
\[
\klabig{ \Phihat_\Ahat^{\nt}(\vhat) }_{j,p}^k = \e^{-\iu j^2 \nt} v_{j,p}^k
\]
and
\[
\klabig{ \Phihat_\Bhat^{\nt}(\vhat) }_{j,p}^k = \klabigg{ \sum_{m=0}^{\infty} \frac{(-\iu \nt)^m}{m!} \!\!\!\! \sum_{p_1+\dots+p_{2m+1} = p} \!\!\!\! \vind_{p_1} \ast \dotsm \ast \vind_{p_{m+1}} \ast \overline{\vind_{p_{m+2}}} \ast \dotsm \ast \overline{\vind_{p_{2m+1}}}}_j^k,
\]
where the notations~\eqref{eq-zbar} and~\eqref{eq-zhatp} are used.
\end{definition*}

\begin{system2}
With these operators, the system that the modulation coefficient functions~$z_{j,p}^k$ of~\eqref{eq-mfe-exp} should satisfy is
\begin{equation}\label{eq-modsystem-exp}\begin{split}
&\e^{-\iu k h} \sum_{\ell=0}^{p-1} \frac{h^{\ell}}{\ell!} \, \frac{\drm^\ell}{\drm\tau^\ell} \, z_{j,p-\ell}^k(\tau) = \klaBig{ \Phihat^{a_1h}_{\Ahat} \circ \Phihat^{b_1h}_{\Bhat} \circ \dotsm \circ \Phihat^{a_sh}_{\Ahat} \circ \Phihat^{b_sh}_{\Bhat} \klabig{\zhat(\tau)} }_{j,p}^k.
\end{split}\end{equation}
This is obtained by inserting the ansatz~\eqref{eq-mfe-exp} in the system~\eqref{eq-modsystem-temp} and comparing the coefficients of~$\ep^p$. Similarly, we get from~\eqref{eq-modsystem-init}
\begin{equation}\label{eq-modsystem-exp-init}
\sum_{k\in\mathbb{Z}} z_{j,p}^k(0) = \begin{cases} \ep^{-1} \psi_j^0, & p=1,\\ 0, & p\ge 2. \end{cases} 
\end{equation}
by~\eqref{eq-init}.
\end{system2}

\subsection{Computation of the modulation coefficient functions}\label{subsec-constr}

In order to (theoretically) compute a solution of the system~\eqref{eq-modsystem-exp} for the coefficients~$z_{j,p}^k$ of~\eqref{eq-mfe-exp}, we rewrite its right-hand side in the following lemma. Recall from Definition~\ref{definition-zbf} that we denote by~$\idhat$ the identity operator on $\C^{\disc\times\Z\times\N}$.

\begin{mylemma}\label{lemma-F}
We have
\[
\Phihat^{a_1h}_{\Ahat} \circ \Phihat^{b_1h}_{\Bhat} \circ \dotsm \circ \Phihat^{a_sh}_{\Ahat} \circ \Phihat^{b_sh}_{\Bhat} \kla{\vhat} = \Phihat^{h}_{\Ahat} + \Fhat
\]
with 
\begin{equation}\label{eq-F}
\Fhat = \sum_{r=1}^s \Phihat^{(a_1+\dots+a_r)h}_{\Ahat} \circ \klabig{\Phihat^{b_r h}_{\Bhat}-\idhat} \circ \Phihat^{a_{r+1}h}_{\Ahat} \circ \Phihat^{b_{r+1}h}_{\Bhat} \circ \dotsm \circ \Phihat^{a_sh}_{\Ahat} \circ \Phihat^{b_sh}_{\Bhat}.
\end{equation}
The operator~$\Fhat$ has the properties that
\begin{enumerate}
\item[(a)] if $v_{j',p'}^{k'}=0$ whenever $\abs{k'}>p' K^2$, then $(\Fhat(\vhat))_{j,p}^k=0$ whenever $\abs{k}>p K^2$,
\item[(b)] $(\Fhat(\vhat))_{j,p}^k$ depends only on $v_{j',p'}^{k'}$ with~$p'\le p-2$,
\item[(c)] if $v_{j',p'}^{k'}$ are polynomials of degree $\deg(v_{j',p'}^{k'})\le p'-1$, 
  then $(\Fhat(\vhat))_{j,p}^k$ is a polynomial of degree $\deg((\Fhat(\vhat))_{j,p}^k)\le p-3$.
\end{enumerate}
\end{mylemma}
\begin{proof}
  The equation~\eqref{eq-F} follows from a telescopic sum, the linearity of $\Phihat_{\Ahat}^{\nt}$ and the consistency assumption~\eqref{eq-consistency}.
  
To prove the properties~(a)--(c) of~$\Fhat$, we observe that $(\Phihat^{\nt}_{\Bhat}(\vhat))_{j,p}^k$ is a sum of products of the form
\[
v_{j_1,p_1}^{k_1}\dotsm v_{j_{m+1},p_{m+1}}^{k_{m+1}} \overline{v_{-j_{m+2},p_{m+2}}^{-k_{m+2}}\dotsm v_{-j_{2m+1},p_{2m+1}}^{-k_{2m+1}}}
\]
with $m\ge 0$, $p_1+\dots+p_{2m+1}=p$ and $k_1+\dots+k_{2m+1}=k$. The same is true for $\Phihat^{\nt}_{\Bhat} - \idhat$, but with $m\ge 1$.

(a) This observation directly implies that~$\Fhat$ has the property~(a).

(b) It also shows that $(\Phihat^{\nt}_{\Bhat}(\vhat))_{j,p}^k$ depends only on $v_{j',p'}^{k'}$ with~$p'\le p$, and that $((\Phihat^{\nt}_{\Bhat}-\idhat)(\vhat))_{j,p}^k$ depends only on $v_{j',p'}^{k'}$ with~$p'\le p-2$. This yields property~(b) of~$\Fhat$.

(c) 
The property~(c) is obtained similarly since the above observation implies that $\deg((\Phihat^{\nt}_{\Bhat} (\vhat))_{j,p}^k) \le p-1$ and $\deg(((\Phihat^{\nt}_{\Bhat} - \idhat)(\vhat))_{j,p}^k) \le p-3$ if $\deg(v_{j',p'}^{k'})\le p'-1$ for all $j'$, $k'$ and $p'$.
\end{proof}

By Lemma~\ref{lemma-F}, the system~\eqref{eq-modsystem-exp} for the coefficients~$z_{j,p}^k$ of~\eqref{eq-mfe-exp} is of the form
\begin{equation}\label{eq-modsystem-exp-comput}
\klabig{\e^{-\iu k h} - \e^{-\iu j^2 h}} z_{j,p}^k(\tau) + \e^{-\iu k h} \sum_{\ell=1}^{p-1} \, \frac{h^{\ell}}{\ell!} \frac{\drm^\ell}{\drm\tau^\ell} \, z_{j,p-\ell}^k(\tau)  = \klabig{\Fhat\kla{\zhat(\tau)}}_{j,p}^k
\end{equation}
with $\Fhat$ from~\eqref{eq-F}. By part~(b) of Lemma~\ref{lemma-F}, the right-hand side $(\Fhat(\zhat))_{j,p}^k$ of~\eqref{eq-modsystem-exp-comput} depends only on~$z_{j',p'}^{k'}$ with~$p'\le p-2$. This fact allows us to use the following iterative algorithm to compute the coefficients $z_{j,p}^k$.

\begin{algorithm}\label{algo-comp}
  For $p=1,2,3,\dots,N$ (in this order) compute $z_{j,p}^k$ for all $j$ and $k$ in the following way: 
  \begin{enumerate}
  \item[(a)] If $k\ne j^2$ and $\abs{k} \le p K^2$: Under the CFL-type step-size restriction~\eqref{eq-cfl}, we have $\abs{(k-j^2)h} \le (p+1)K^2 h < 2\pi$, and hence $\e^{-\iu k h}\ne \e^{-\iu j^2 h}$. We can thus solve~\eqref{eq-modsystem-exp-comput} for $z_{j,p}^k$ in this case.
  \item[(b)] If $k\ne j^2$ and $\abs{k} > p K^2$: In this case we set $z_{j,p}^k=0$ without introducing any error, since the right-hand side of~\eqref{eq-modsystem-exp-comput} and the derivatives on the left-hand side vanish by Lemma~\ref{lemma-F}~(a).
  \item[(c)] If $k=j^2$ and $p\le N-1$: In this case, the factor in front of~$z_{j,p}^{k}$ in~\eqref{eq-modsystem-exp-comput} vanishes. We therefore use~\eqref{eq-modsystem-exp-comput} with~$p+1$ instead of~$p$ to compute the derivative~$\frac{\drm}{\drm\tau} \, z_{j,p}^{k}$. To compute the initial value $z_{j,p}^{k}(0)$, we solve~\eqref{eq-modsystem-exp-init} (where~$z_{j,p}^{k'}(0)$ for $k'\ne k=j^2$ has already been computed in steps~(a) and~(b)). 
    \item[(d)] If $k=j^2$ and $p=N$: In this case, $\frac{\drm}{\drm\tau} \, z_{j,p}^{k}$ does not appear in any of the equations~\eqref{eq-modsystem-exp-comput} with $p\le N$. We therefore set it to zero and compute the initial value $z_{j,p}^{k}(0)$ from~\eqref{eq-modsystem-exp-init}. 
  \end{enumerate}
\end{algorithm}

During the above computation, all formally infinite sums in~\eqref{eq-modsystem-exp-comput} are actually finite sums. By part~(c) of Lemma~\ref{lemma-F}, the constructed modulation coefficient functions~$z_{j,p}^k$ are polynomials of degree
\begin{equation}\label{eq-degree}
  \deg\klabig{z_{j,p}^{k}} \le p-1 \le N-1.
\end{equation}
In addition, we have
\begin{equation}\label{eq-offdiag-zero}
  z_{j,1}^k = z_{j,2}^k = 0 \myfor k\ne j^2
\end{equation}
by part~(b) of Lemma~\ref{lemma-F} and steps~(a) and~(b) of Algorithm~\ref{algo-comp}. 

We finally note that the above construction solves the systems~\eqref{eq-modsystem-exp} and~\eqref{eq-modsystem-exp-init} exactly for $p=1,\dots,N$. However, as the expansion~\eqref{eq-mfe-exp} is a truncated expansion, this does not give an exact solution for the system~\eqref{eq-modsystem}, and hence the approximation~\eqref{eq-mfe-approx} of the numerical solution by the modulated Fourier expansion is not exact. The defect in the system~\eqref{eq-modsystem} and the error in the approximation~\eqref{eq-mfe-approx} will be studied in the following Section~\ref{sec-estimates}.

\section{Estimates of the modulated Fourier expansion}\label{sec-estimates}

In this section, we derive estimates of the modulated Fourier expansion as constructed in the previous section. As there, we work under the assumptions of Theorem~\ref{thm-main}. By~$\lesssim$, we denote an inequality up to a constant that is independent of the small parameter~$\ep$ of~\eqref{eq-init}, the time step-size~$h$ and the spatial discretization parameter~$K$, but may depend on~$N$ from Theorem~\ref{thm-main}, $c_0$ from the step-size restriction~\eqref{eq-cfl}, the number of stages~$s$ of the method~\eqref{eq-method} and an upper bound on the absolute values of the coefficients $b_1,\dots,b_s$ of the method.

\subsection{Bounds of the modulation coefficient functions}

We derive bounds of the modulation coefficient functions $z_{j,p}^k$ as constructed in Algorithm~\ref{algo-comp}. We begin with the following lower bound on the coefficient $\e^{-\iu k h} - \e^{-\iu j^2 h}$ in the system~\eqref{eq-modsystem-exp-comput}, which holds thanks to the step-size restriction~\eqref{eq-cfl}. 

\begin{mylemma}\label{lemma-cfl}
Let $j\in\disc$ and $k\in\mathbb{Z}$ with $k\ne j^2$ and $\abs{k} \le N K^2$. Then, we have
  \[
\absbig{ \e^{-\iu k h} - \e^{-\iu j^2 h} }^{-1} \lesssim \abs{k-j^2}^{-1} h^{-1}.
\]
\end{mylemma}
\begin{proof}
Let $\gamma = (j^2-k)h$ and note that $\abs{\gamma} \le  (1+N)K^2h \le c_0 < 2\pi$ by~\eqref{eq-cfl}. We have
  \[
\absbig{ \e^{-\iu k h} - \e^{-\iu j^2 h} } = \absbig{\e^{\iu \gamma} - 1} = 2 \, \absbig{\sin\klabig{\tfrac12\gamma}} \ge \abs{\gamma}\, \absbigg{\frac{\sin(\tfrac12 c_0)}{\tfrac12 c_0}},
\]
where we use in the final estimate that $\sinc(x) = \sin(x)/x$ is a monotonically decreasing and positive function on $[0,\frac12 c_0]$.
\end{proof}

We next derive bounds on the nonlinearity~$\Fhat$ in~\eqref{eq-modsystem-exp-comput}. We do this in the seminorms~$\abs{\cdot}_{1,p}$ defined in~\eqref{eq-seminorm}.

\begin{mylemma}\label{lemma-F-bounds}
  Fix $p\in\N$ and $L\in\N_0$, and let $\vhat=\vhat(\tau)\in\mathbb{C}^{\disc\times\Z\times\mathbb{N}}$ depend on time~$\tau$ with
  \[
    \absBig{\sfrac{\drm^\ell}{\drm\tau^\ell}\vhat(\tau)}_{1,p-2}\le M \myfor \ell=0,1,\dots,L.
  \]
  (For $p\le 2$, this is an empty condition.)
  Then, we have
  \[
\absBig{\sfrac{\drm^\ell}{\drm\tau^\ell}\Fhat\klabig{\vhat(\tau)}}_{1,p} \le C h \myfor \ell=0,1,\dots,L
  \]  with~$C$ depending only on the bound~$M$, the number of derivatives~$L$, the number of stages~$s$ and an upper bound on the absolute values of the coefficients $b_1,\dots,b_s$ of the method~\eqref{eq-method}.
\end{mylemma}
\begin{proof}
(a) We first consider the case~$\ell=0$ and recall that~$\Fhat$ is given by~\eqref{eq-F} in Lemma~\ref{lemma-F}, and hence we have to deal with compositions of~$\Phihat^{\nt}_{\Ahat}$, $\Phihat^{\nt}_{\Bhat}$ and~$\Phihat^{\nt}_{\Bhat}-\idhat$. We use that
\[
\absbig{\Phihat^{\nt}_{\Ahat}(\vhat)}_{1,q} = \abs{\vhat}_{1,q}.
\]
In addition, we use that (with the notation~\eqref{eq-zbar})
\begin{align*}
  \absbig{\Phihat^{\nt}_{\Bhat}(\vhat)}_{1,q} &\le \normvbigg{ \sum_{m=0}^\infty \frac{\abs{\nt}^m}{m!}  \!\! \sum_{p_1+\dots+p_{2m+1} \le q} \!\! \absbig{ \vind_{p_1} \ast \dotsm \ast \vind_{p_{m+1}} \ast \overline{\vind_{p_{m+2}}} \ast \dotsm \ast \overline{\vind_{p_{2m+1}}} } }_1\\
                                               &\le \sum_{m=0}^\infty \frac{\abs{\nt}^m}{m!} \normvbigg{ \klabigg{\sum_{p_1=1}^{q-2m} \abs{\vhat_{p_1}}} \ast \dotsm \ast \klabigg{\sum_{p_{2m+1}=1}^{q-2m} \abs{\overline{\vhat_{p_{2m+1}}}}} }_1\\
  & \le \abs{\vhat}_{1,q} \cdot e^{C_1^2 \abs{\nt} \cdot \abs{\vhat}_{1,q}^2},
\end{align*}
where we use the property~\eqref{eq-algebra} of the convolution in the final estimate (and~$C_1$ is the constant from there).
With the same estimates, we also get
\[
  \absbig{\klabig{\Phihat^{\nt}_{\Bhat}-\idhat}(\vhat)}_{1,q} \le
  \abs{\vhat}_{1,q-2} \cdot \klabig{ e^{C_1^2 \abs{\nt} \cdot \abs{\vhat}_{1,q-2}^2} -1 },
\]
since the sum then starts with $m=1$ instead of $m=0$.
The claimed estimate of~$\Fhat$ follows from these estimates on the individual factors in~$\Fhat$.

(b) To prove the claimed estimate for derivatives of~$\Fhat(\zhat)$, we proceed similarly. We use that
\[
\absBig{\sfrac{\drm^\ell}{\drm\tau^\ell} \Phihat^{\nt}_{\Ahat}\klabig{\vhat(\tau)}}_{1,q} = \absBig{\sfrac{\drm^\ell}{\drm\tau^\ell}\vhat(\tau)}_{1,q}.
\]
To derive a bound on $\frac{\drm^\ell}{\drm\tau^\ell}\Phihat^{\nt}_{\Bhat}(\vhat(\tau))$, we use the same estimates as in part~(a) of the proof and in addition
\begin{multline*}
\sum_{p_1+\dots+p_{2m+1} \le q} \absBig{ \sfrac{\drm^\ell}{\drm\tau^\ell} \vind_{p_1} \ast \dotsm \ast \vind_{p_{m+1}} \ast \overline{\vind_{p_{m+2}}} \ast \dotsm \ast \overline{\vind_{p_{2m+1}}} }\\
  \le (2m+1)^\ell \klabigg{ \max_{\ell'=0,\dots,\ell} \sum_{p_1=1}^{q-2m} \absBig{\sfrac{\drm^{\ell'}}{\drm\tau^{\ell'}}  \vind_{p_1} }} \ast \dotsm \ast \klabigg{ \max_{\ell'=0,\dots,\ell} \sum_{p_{2m+1}=1}^{q-2m} \absBig{\sfrac{\drm^{\ell'}}{\drm\tau^{\ell'}} \overline{\vind_{p_{2m+1}}} } }.
\end{multline*}
Using also $(2m+1)^\ell\le 3^{m\ell}$ this yields, with $\gamma_q = \gamma_q(\tau) = \max_{\ell'=0,\dots,\ell} \absBig{\tfrac{\drm^{\ell'}}{\drm\tau^{\ell'}} \vhat(\tau)}_{1,q}$,
\[
  \absBig{\sfrac{\drm^\ell}{\drm\tau^\ell}\Phihat^{\nt}_{\Bhat}\klabig{\vhat(\tau)} }_{1,q}
  \le \gamma_q e^{3^\ell C_1^2 \abs{\nt} \,\gamma_q^2}
\]
and 
\[
  \absBig{\sfrac{\drm^\ell}{\drm\tau^\ell}\klabig{\Phihat^{\nt}_{\Bhat}-\idhat}\klabig{\vhat(\tau)} }_{1,q}
  \le \gamma_{q-2} \klabig{e^{3^\ell C_1^2 \abs{\nt} \,\gamma_{q-2}^2} - 1}.
\]
In this way we get the statement of the lemma also for derivatives of~$\Fhat(\vhat(\tau))$.
\end{proof}

With the previous lemmas, we can now prove the following bounds on the sequences
\[
\zind_p(\tau) = \klabig{z_{j,p}^k(\tau)}_{j\in\disc, k\in\Z}
\]
from Algorithm~\ref{algo-comp} (recall the notation~\eqref{eq-zhatp}). 
We use here the norm $\normv{\cdot}_\sigma$ of~\eqref{eq-normv} and the rescalings
\[
\Labf \vbf = \klabig{\skla{k-j^2} v_j^k}_{j\in\disc,k\in\Z} \myand 
\Kbf \vbf = \klabig{\skla{k} v_j^k}_{j\in\disc,k\in\Z}
\]
of sequences $\vbf = \klabig{v_j^k}_{j\in\disc,\,k\in\mathbb{Z}}\in\mathbb{C}^{\disc\times\mathbb{Z}}$, where~$\skla{\cdot}$ is defined in~\eqref{eq-skla}. 

\begin{mylemma}\label{lemma-size}
For $p=1,\dots,N$ and all $\ell\in\N_0$, we have for the modulation coefficient functions of Algorithm~\ref{algo-comp} 
\[
\normvBig{\Kbf^{1/2} \sfrac{\drm^\ell}{\drm\tau^\ell} \zind_p(\tau)}_0 \lesssim \normvBig{\Labf \sfrac{\drm^\ell}{\drm\tau^\ell} \zind_p(\tau)}_{1} \lesssim 1.
\]
\end{mylemma}
\begin{proof}
The first inequality follows immediately from \[\skla{k}\lesssim \skla{k-j^2}\skla{j^2} \le \skla{k-j^2}^2 \skla{j}^2.\]

For the proof of the second inequality, we recall that the modulation coefficient functions~$z_{j,p}^k$ are constructed in Algorithm~\ref{algo-comp} for $p=1,2,\dots,N$ consecutively. We therefore proceed by induction on $p=1,2,\dots,N$, similarly to the proof of \cite[Lemma~5.3]{Gauckler2017a}. 

We fix~$p$, and we assume that we have shown the bounds for the polynomials~$\zind_{p'}$ with $p'\le p-1$ and all their derivatives (this is true for~$p=1$). 
To show the claimed bound for $\zind_p=(z_{j,p}^k)_{j\in\disc,k\in\Z}$, we split this sequence into a sequence containing only functions $z_{j,p}^k$ with $k\ne j^2$ and a sequence containing the remaining functions with $k=j^2$, and we prove the claimed bound for both types of functions separately.

We first consider the sequence consisting only of functions $z_{j,p}^k$ with $k\ne j^2$. These are constructed as described in steps~(a) and~(b) of Algorithm~\ref{algo-comp}.
By induction, we get from Lemma~\ref{lemma-F-bounds} for the right-hand side of~\eqref{eq-modsystem-exp-comput} that
\[
\normvbig{\klabig{\Fhat\kla{\zhat(\tau)}}_p }_1 \le \absbig{\Fhat\kla{\zhat(\tau)} }_{1,p} \lesssim h ,
\]
where we write $(\Fhat(\zhat))_p=((\Fhat(\zhat))_{j,p}^k)_{j\in\disc,k\in\Z}$ in accordance with the notation~\eqref{eq-zhatp}.
We also get by induction for the derivatives on the left-hand side of~\eqref{eq-modsystem-exp-comput} that
\[
\normvBig{\sfrac{\drm^{\ell}}{\drm\tau^{\ell}} \zhat_{p-\ell}(\tau) }_1 \lesssim 1 \myfor \ell=1,2,\ldots . 
\]
Using Lemma~\ref{lemma-cfl} to bound $\e^{-\iu k h}-\e^{-\iu j^2 h}$ and the definition of~$\Labf$, this implies the claimed estimate with $\ell=0$ for the sequence of modulation coefficient functions with $k\ne j^2$. The estimate for their higher derivatives follows similarly using in addition that the polynomials have only a bounded number of nonzero derivatives by~\eqref{eq-degree}.

We finally consider functions $z_{j,p}^k$ with indices $k=j^2$, which are constructed as described in step~(c) of Algorithm~\ref{algo-comp}. By induction, we have
\[
\normvbig{\klabig{\Fhat\kla{\zhat(\tau)}}_{p+1} }_1 \lesssim h ,
\]
and we also get for the derivatives on the left-hand side of~\eqref{eq-modsystem-exp-comput} that
\[
\normvBig{\sfrac{\drm^{\ell}}{\drm\tau^{\ell}} \zhat_{p+1-\ell}(\tau) }_1 \lesssim 1 \myfor \ell = 2,3,\ldots.  
\]
This implies the claimed estimate for the first  derivative of the modulation coefficient functions with $k=j^2$. The bound for higher derivatives is obtained similarly. To also get the bound for the functions themselves, we estimate the initial values given by~\eqref{eq-modsystem-exp-init} using~\eqref{eq-init} and the already proven bounds for the functions~$z_{j,p}^k$ with $k\ne j^2$. The same argument also yields the claimed estimates for the functions constructed in step~(d) of Algorithm~\ref{algo-comp}. 
\end{proof}

\subsection{Bounds of the defect}\label{subsec-defect}

Recall from Section~\ref{subsec-constr} that we cannot expect that the modulation functions~\eqref{eq-mfe-exp} with coefficients computed with Algorithm~\ref{algo-comp} satisfy the system~\eqref{eq-modsystem} exactly. In this section, we first derive a formula for the defect in~\eqref{eq-modsystem} and then provide estimates for it.

\begin{mylemma}\label{lemma-defect-formula}
The functions~$z_j^k$ of~\eqref{eq-mfe-exp} whose coefficients~$z_{j,p}^k$ have been computed with Algorithm~\ref{algo-comp} satisfy
\begin{equation}\label{eq-modsystem-defect}
\e^{-\iu k h} z_j^k(\tau+\ep h) + d_j^k(\tau) = \klaBig{ \Phbf^{a_1h}_{\Abf} \circ \Phbf^{b_1h}_{\Bbf} \circ \dotsm \circ \Phbf^{a_sh}_{\Abf} \circ \Phbf^{b_sh}_{\Bbf} \klabig{\zbf(\tau)} }_j^k
\end{equation}
with the defect
\begin{equation}\label{eq-defect}
d_j^k(\tau) = \sum_{p=N+1}^\infty \ep^p \klabig{\Fhat\kla{\zhat(\tau)}}_{j,p}^k - \e^{-\iu k h} \sum_{\ell=1}^\infty \frac{\ep^\ell h^{\ell}}{\ell!}\sum_{p=N-\ell+1}^N \ep^{p} \, \frac{\drm^\ell}{\drm\tau^\ell} \, z_{j,p}^k(\tau) .
\end{equation}
\end{mylemma}
\begin{proof}
Recall that Algorithm~\ref{algo-comp} solves~\eqref{eq-modsystem-exp-comput} exactly for $p\le N$. Multiplying~\eqref{eq-modsystem-exp-comput} with~$\ep^p$ and summing from $p=1$ to $p=N$, we thus get for $z_j^k$ from~\eqref{eq-mfe-exp}
\[
\e^{-\iu k h} \sum_{p=1}^N \ep^p \sum_{\ell=0}^{p-1} \frac{h^{\ell}}{\ell!} \, \frac{\drm^\ell}{\drm\tau^\ell} \, z_{j,p-\ell}^k(\tau)  = \e^{-\iu j^2 h} z_{j}^k(\tau) + \sum_{p=1}^N \ep^p \klabig{\Fhat\kla{\zhat(\tau)}}_{j,p}^k.
\]
On the left-hand side, we then use that (with $z_{j,p'}^k=0$ for $p'\le 0$)
\begin{align*}
  \sum_{p=1}^N \ep^p \sum_{\ell=0}^{p-1} \frac{h^{\ell}}{\ell!} \frac{\drm^\ell}{\drm\tau^\ell} \, z_{j,p-\ell}^k(\tau) & = \sum_{\ell=0}^\infty \frac{\ep^\ell h^{\ell}}{\ell!} \, \frac{\drm^\ell}{\drm\tau^\ell} \klabigg{ z_j^k(\tau) - \sum_{p=N-\ell+1}^N \ep^{p} z_{j,p}^k(\tau)}\\
  & = z_j^k(\tau+\ep h) - \sum_{\ell=0}^\infty \frac{\ep^\ell h^{\ell}}{\ell!}\sum_{p=N-\ell+1}^N \ep^{p} \, \frac{\drm^\ell}{\drm\tau^\ell} \, z_{j,p}^k(\tau).
\end{align*}
On the right-hand side, we use that
\begin{multline*}
\e^{-\iu j^2 h} z_{j}^k(\tau) + \sum_{p=1}^N \ep^p \klabig{\Fhat\kla{\zhat(\tau)}}_{j,p}^k\\ = \klaBig{ \Phbf^{a_1h}_{\Abf} \circ \Phbf^{b_1h}_{\Bbf} \circ \dotsm \circ \Phbf^{a_sh}_{\Abf} \circ \Phbf^{b_sh}_{\Bbf} \klabig{\zbf(\tau)} }_j^k - \sum_{p=N+1}^\infty \ep^p \klabig{\Fhat\kla{\zhat(\tau)}}_{j,p}^k
\end{multline*}
by Lemma~\ref{lemma-F}. This yields the statement of the lemma.
\end{proof}

\begin{mylemma}\label{lemma-defect}
  For the defect~\eqref{eq-defect}, we have
  \[
\normv{\dbf(\tau)}_1 \lesssim \ep^{N+1} h 
  \]
  and
  \[
    \normv{\Kbf\dbf(\tau)}_1 \lesssim \ep^{N+1}.
  \]
\end{mylemma}
\begin{proof}
  The defect~$\dbf$ is given by~\eqref{eq-defect}. 
  The first estimate follows from the bounds of Lemma~\ref{lemma-size} on the modulation coefficient functions $z_{j,p}^k$, the bounds of Lemma~\ref{lemma-F-bounds} on~$\Fhat$ (which are uniform in $p$) and the fact that the functions~$z_{j,p}^k$ are polynomials of degree (at most) $N-1$ by~\eqref{eq-degree}.

  For the second estimate, we use that $\skla{k}\lesssim p K^2$ whenever $(\Fhat(\zhat(\tau)))_{j,p}^k\ne 0$ by Lemma~\ref{lemma-F}~(a) and step~(b) of Algorithm~\ref{algo-comp}. By the step-size restriction~\eqref{eq-cfl}, we thus have $\skla{k}\lesssim p (N+1)^{-1} h^{-1} c_0$ whenever $(\Fhat(\zhat(\tau)))_{j,p}^k\ne 0$. From Lemmas~\ref{lemma-F-bounds} and~\ref{lemma-size}, we then get
  \[
\normvbigg{\Kbf \!\! \sum_{p=N+1}^\infty \ep^p \klabig{\Fhat\kla{\zhat(\tau)}}_{p}}_1 \lesssim h^{-1} \!\! \sum_{p=N+1}^\infty \ep^p p \, \normvbig{ \klabig{\Fhat\kla{\zhat(\tau)}}_{p}}_1 \lesssim \sum_{p=N+1}^\infty \ep^p p \lesssim \ep^{N+1}.
\]
To deal with the derivatives in the defect, we similarly use that $\skla{k}\le N K^2\le h^{-1} c_0$ whenever $z_{j,p}^k\ne 0$ by step~(b) of Algorithm~\ref{algo-comp} and the step-size restriction~\eqref{eq-cfl}. By Lemma~\ref{lemma-size}, this yields
  \[
\normvbigg{\Kbf \sum_{\ell=1}^\infty \frac{\ep^\ell h^{\ell}}{\ell!}\sum_{p=N-\ell+1}^N \ep^{p} \, \frac{\drm^\ell}{\drm\tau^\ell} \, \zhat_p(\tau)}_1 \lesssim \e^{N+1}.
  \]
This finally implies the claimed estimate of~$\Kbf\dbf$.
\end{proof}

We remark that the estimates of the defect are simpler than in~\cite[Section~5.6]{Gauckler2017a} which is due to the fact that we use here the stronger CFL-type step-size restriction~\eqref{eq-cfl}.

\subsection{Approximation of the numerical solution}\label{subsec-error}

The modulation functions~$z_j^k$ of~\eqref{eq-mfe-exp} with the coefficients~$z_{j,p}^k$ of Algorithm~\ref{algo-comp} give an approximation~$\widetilde\psi^n=\widetilde\psi(\cdot,t_n)$ to the numerical solution~$\psi^n$ of~\eqref{eq-method} via~\eqref{eq-mfe} and~\eqref{eq-mfe-approx}. 
This approximation satisfies the defining relation~\eqref{eq-method} of the method up to a small defect, as shown in the following lemma. 

\begin{mylemma}\label{lemma-eqpsitilde}
We have 
\begin{equation}\label{eq-method-approx-init}
  \widetilde{\psi}^0 = \psi^0
\end{equation}
and 
\begin{equation}\label{eq-method-approx}
\widetilde{\psi}^{n+1} + \delta^n = \phi^{a_1h}_{A} \circ \phi^{b_1h}_{B} \circ \phi^{a_2h}_{A} \circ \phi^{b_2h}_{B} \circ \dotsm \circ \phi^{a_sh}_{A} \circ \phi^{b_sh}_{B} \klabig{\widetilde{\psi}^n}
\end{equation}
with
\[
\delta^n(x) = \sum_{j\in\disc} \sum_{k\in\Z} d_j^k(\ep t_n) \e^{-\iu kt_n} \e^{\iu j x},
\]
where $d_j^k$ is the defect~\eqref{eq-defect}.
\end{mylemma}
\begin{proof}
The modulation functions~$z_j^k$ of~\eqref{eq-mfe-exp}
satisfy the system~\eqref{eq-modsystem} up to a defect~$d_j^k$, see~\eqref{eq-modsystem-defect} in Lemma~\ref{lemma-defect-formula}, and satisfy~\eqref{eq-modsystem-init} exactly, see steps~(c) and~(d) of Algorithm~\ref{algo-comp}.
  The property~\eqref{eq-method-approx-init} is thus clear by construction, and~\eqref{eq-method-approx} follows by  
multiplying~\eqref{eq-modsystem-defect} with $\e^{-\iu k t_n} \e^{\iu j x}$ and summing over $k\in\Z$ and $j\in\disc$.
\end{proof}

Before studying the size of the error $\psi^n-\widetilde{\psi}^n$, we consider~$\psi^n$ and~$\widetilde{\psi}^n$ separately. We introduce, for $r=1,\dots,s$, the abbreviations
\[
  \psi^{n+\frac{r}{s}} = \phi^{a_{s+1-r}h}_{A} \circ \phi^{b_{s+1-r}h}_{B} \klabig{\psi^{n+\frac{r-1}{s}}}
\]
for the intermediate values of the numerical method~\eqref{eq-method} (note that this is well-defined: $\psi^{n+\frac{s}{s}}=\psi^{n+1}$). We also introduce, but only for $r=1,\dots,s-1$, the abbreviations
\[
  \widetilde{\psi}^{n+\frac{r}{s}} = \phi^{a_{s+1-r}h}_{A} \circ \phi^{b_{s+1-r}h}_{B} \klabig{\widetilde\psi^{n+\frac{r-1}{s}}}
\]
for the intermediate values of the approximation by a modulated Fourier expansion (note that an extension of this definition to $r=s$ would be in conflict with the defect in~\eqref{eq-method-approx}). 

\begin{mylemma}\label{lemma-size-psitilde}
  We have, for $r=0,1,\dots,s$,
  \[
    \normbig{\psi^{n+\frac{r}{s}}}_1 \le 2\ep \myand \normbig{\widetilde\psi^{n+\frac{r}{s}}}_1 \lesssim \ep \myfor 0\le t_n=n h\le\ep^{-1}.
\]
\end{mylemma}
\begin{proof}
(a) We first prove the estimate of $\psi^{n+\frac{r}{s}}$. To this end, we use that (see~\eqref{eq-flows})
  \begin{equation}\label{eq-boundsflows}
\norm{\phi_A^{\nt}(\eta)}_1 = \norm{\eta}_1 \myand \norm{\phi_B^{\nt}(\eta)}_1 \le \exp\klabig{C^2 \norm{\eta}_1^2 \abs{\nt}} \norm{\eta}_1,
  \end{equation}
  where~$\norm{\cdot}_1$ is the Sobolev $H^1(\mathbb{T})$-norm~\eqref{eq-norm} and where~$C$ is a constant such that $\norm{\mathcal{Q}^K(u v)}_1\le C \norm{u}_1\norm{v}_1$ for all $u,v\in H^1(\mathbb{T})$ (such a constant exists, see, for example, \cite[Lemma~4.2]{Hairer2008}). These estimates imply, for $r=1,\dots,s$,
  \[
\normbig{\psi^{n+\frac{r}{s}}}_1 \le \exp\klaBig{C^2 \normbig{\psi^{n+\frac{r-1}{s}}}_1^2 \abs{b_{s+1-r}} \, h} \normbig{\psi^{n+\frac{r-1}{s}}}_1.
\]
For sufficiently small~$\ep$ such that $\exp \klabig{4 C^2 \ep \max_{r'=1,\dots,s} \abs{b_{r'}} \, s \, (1+\ep h)}\le 2$ and as long as $t_n=n h\le \ep^{-1}$, we get by solving this recursion
\[
\normbig{\psi^{n+\frac{r}{s}}}_1 \le \exp \klaBig{4 C^2 \ep^2 \max_{r'=1,\dots,s} \abs{b_{r'}} (n s + r)h} \ep \le 2\ep,
\]
where the smallness~\eqref{eq-init} of $\psi^0$ is used.

(b) For the estimate of $\widetilde\psi^{n+\frac{r}{s}}$, we first consider the case $r=0$. In this case, the claimed estimate is an immediate consequence of Lemma~\ref{lemma-size} on the size of the modulation coefficient functions~$z_{j,p}^k$ that make up the modulation functions~$z_j^k$ (see in particular~\eqref{eq-mfe}, \eqref{eq-mfe-approx} and~\eqref{eq-mfe-exp}). In the case $0<r<s$, we proceed as in step~(a), but we solve the recursion
\[
\normbig{\widetilde\psi^{n+\frac{r}{s}}}_1 \le \exp\klaBig{C^2 \normbig{\widetilde\psi^{n+\frac{r-1}{s}}}_1^2 \abs{b_{s+1-r}} \, h} \normbig{\widetilde\psi^{n+\frac{r-1}{s}}}_1
\]
only up to $\widetilde\psi^n$. This yields the claimed estimate for the intermediate values~$\widetilde\psi^{n+\frac{r}{s}}$.
\end{proof}

With Lemmas~\ref{lemma-eqpsitilde} and~\ref{lemma-size-psitilde} we can now prove the following result on the error $\psi^n-\widetilde\psi^n$.

\begin{mylemma}\label{lemma-error-psitilde}
  We have
  \[
  \norm{\psi^n-\widetilde\psi^n}_1 \lesssim \ep^{N} \myfor 0\le t_n=n h\le\ep^{-1}.
\]
\end{mylemma}
\begin{proof}
  We subtract~\eqref{eq-method-approx} from~\eqref{eq-method}. We then use the Lipschitz properties
  \[
    \norm{\phi_A^{\nt}(\eta)-\phi_A^{\nt}(\widetilde\eta)}_1 = \norm{\eta-\widetilde\eta}_1
  \]
  and
  \begin{align*}
    \norm{\phi_B^{\nt}(\eta)-\phi_B^{\nt}(\widetilde\eta)}_1 &\le \sum_{m=0}^\infty \frac{\abs{\nt}^m}{m!} (2m+1) C^{2m} \max\klabig{\norm{\eta}_1,\norm{\widetilde\eta}_1}^{2m} \norm{\eta-\widetilde\eta}_1 \\
    &\le \exp\klaBig{3 C^2 \max\klabig{\norm{\eta}_1,\norm{\widetilde\eta}_1}^2 \abs{\nt}} \norm{\eta-\widetilde\eta}_1
  \end{align*}
  of the flows, which can be obtained similarly to the boundedness properties~\eqref{eq-boundsflows} (with the same constant~$C$ as there). By Lemma~\ref{lemma-size-psitilde} on the size of~$\psi$ and~$\widetilde\psi$, these estimates imply, for $r=1,\dots,s-1$,
  \[
\normbig{\psi^{n+\frac{r}{s}}-\widetilde\psi^{n+\frac{r}{s}}}_1 \le \exp\kla{C' \ep^2 h} \normbig{\psi^{n+\frac{r-1}{s}}-\widetilde\psi^{n+\frac{r-1}{s}}}_1
\]
with some constant~$C'$ that is independent of~$n$ as long as $t_n\le \ep^{-1}$. Using in addition~\eqref{eq-method-approx} from Lemma~\ref{lemma-eqpsitilde}, we also get
\[
\normbig{\psi^{n+1}-\widetilde\psi^{n+1}}_1 \le \exp\kla{C' \ep^2 h} \normbig{\psi^{n+\frac{s-1}{s}}-\widetilde\psi^{n+\frac{s-1}{s}}}_1 + \norm{\delta^n}_1,
\]
and hence
\[
\normbig{\psi^{n+1}-\widetilde\psi^{n+1}}_1 \le \exp\kla{C' \ep^2 s h} \normbig{\psi^{n}-\widetilde\psi^{n}}_1 + \norm{\delta^n}_1.
\]
Solving this recursion yields
\[
\normbig{\psi^n-\widetilde\psi^{n}}_1 \le \exp\klabig{C' \ep^2 s \, t_n} \normbig{\psi^0-\widetilde\psi^0}_1 + \sum_{\ell=0}^{n-1} \exp\klabig{C' \ep^2 s \, t_{n-1-\ell}} \norm{\delta^\ell}_1. 
\]
To get the statement of the Lemma, we use that $\psi^0=\widetilde\psi^0$ by Lemma~\ref{lemma-eqpsitilde} and that $\norm{\delta^\ell}_1 \le \normv{\dbf(\ep t_\ell)}_1 \lesssim \ep^{N+1} h$ by Lemma~\ref{lemma-defect}.
\end{proof}

\section{An almost-invariant of the modulated Fourier expansion}\label{sec-invariants}

We work under the same assumptions as in Sections~\ref{sec-mfe} and~\ref{sec-estimates}. In particular, we assume that the CFL-type step-size restriction~\eqref{eq-cfl} holds and that the initial value is small~\eqref{eq-init}. Again, we denote by~$\lesssim$ an inequality up to a constant that is independent of the small parameter~$\ep$, the time step-size~$h$ and the spatial discretization parameter~$K$.

Our goal in this section is to derive an almost-invariant of the modulated Fourier expansion (as constructed in Section~\ref{sec-mfe} with Algorithm~\ref{algo-comp}) and to establish its relation with the squared $H^1(\mathbb{T})$-norm of the numerical solution. This is based on the observation that the system~\eqref{eq-modsystem} for the modulation functions is a splitting method.

\begin{mylemma}\label{lemma-flows}
  The operators $\Phbf_\Abf^\tau$ and $\Phbf_\Bbf^\tau$ given by~\eqref{eq-Phbf-A} and~\eqref{eq-Phbf-B} are the time-$\tau$ flows of
  \[
    \iu \, \sfrac{\drm}{\drm\tau} \ubf = \Abf\ubf \myand \iu \, \sfrac{\drm}{\drm\tau} \ubf = \Bbf(\ubf)
\]
with the linear operator $(\Abf\ubf)_j^k=j^2 u_j^k$ and the nonlinear operator $\Bbf(\ubf)=\ubf \ast \ubf \ast \overline{\ubf}$. Both of these differential equations have
\[
 \Ec(\vbf) = \sum_{j\in\disc} \sum_{k\in\Z} (k+1) \absbig{v_j^k}^2
\]
as a conserved quantity, i.e., $\Ec(\Phbf_\Abf^\tau(\vbf)) = \Ec(\vbf)$ and $\Ec(\Phbf_\Bbf^\tau(\vbf)) = \Ec(\vbf)$ for all times~$\tau$.
\end{mylemma}
\begin{proof}
  It is clear that~$\Phbf_\Abf^\tau$ is the flow of the differential equation with~$\Abf$ and that~$\Ec$ is conserved along its solutions.

  For $\Phbf_\Bbf^\tau$ as given by~\eqref{eq-Phbf-B} we first verify that it is the flow of the differential equation with~$\Bbf$. We compute
  \[
\iu \, \sfrac{\drm}{\drm\tau} \Phbf_\Bbf^\tau(\vbf) = \sum_{m=1}^\infty \frac{(-\iu \tau)^{m-1}}{(m-1)!} \; \underbrace{\vbf \ast \dotsm \ast \vbf}_{\text{$m+1$ times}} \ast \underbrace{\overline{\vbf} \ast \dotsm \ast \overline{\vbf}}_{\text{$m$ times}}
\]
and
\begin{multline*}
\Phbf_\Bbf^\tau(\vbf) \ast \Phbf_\Bbf^\tau(\vbf) \ast \overline{\Phbf_\Bbf^\tau(\vbf)}\\ = \sum_{m=0}^\infty \klabigg{\sum_{m_1+m_2+m_3=m} \frac{(-1)^{m_3}}{m_1!m_2!m_3!}} (-\iu \tau)^m \underbrace{\vbf \ast \dotsm \ast \vbf}_{\text{$m+2$ times}} \ast \underbrace{\overline{\vbf} \ast \dotsm \ast \overline{\vbf}}_{\text{$m+1$ times}}.
\end{multline*}
The statement follows then from the fact that
\[
  \sum_{m_1+m_2+m_3=m} \frac{(-1)^{m_3}m!}{m_1!m_2!m_3!}=(1+1-1)^m=1
\]
by the multinomial theorem.

To show that~$\Ec$ is conserved along solutions $\ubf(\tau)=\Phbf_\Bbf^\tau(\vbf)$ of the differential equation $\iu \, \frac{\drm}{\drm\tau} \ubf = \Bbf(\ubf)$ (with initial value $\ubf(0)=\vbf$) we compute
\begin{align*}
  - \iu \, \sfrac{\drm}{\drm\tau} \Ec\klabig{\ubf(\tau)} 
  & = \sum_{j\in\disc, k\in\Z} (k+1) \klaBig{ u_{j}^{k} \; \overline{ \iu \sfrac{\drm}{\drm\tau} u_j^k} - \overline{u_{j}^{k}} \; \iu \sfrac{\drm}{\drm\tau} u_j^k}\\
  & = \sum_{j\in\disc, k\in\Z} (k+1) u_{j}^{k} \sum_{\substack{j_2-j_3-j_4\equiv -j\\ k_2-k_3-k_4=-k}} u_{j_2}^{k_2} \overline{u_{j_3}^{k_3}} \overline{u_{j_4}^{k_4}}\\
    &\qquad + \sum_{j\in\disc, k\in\Z} (-k-1) \overline{u_{j}^{k}} \sum_{\substack{j_1+j_2-j_3\equiv j\\ k_1+k_2-k_3=k}} u_{j_1}^{k_1} u_{j_2}^{k_2} \overline{u_{j_3}^{k_3}}.
\end{align*}
By symmetry, the right-hand side is equal to
\begin{align*}
  \frac12 \sum_{\substack{j_1+j_2-j_3-j_4\equiv 0\\ k_1+k_2-k_3-k_4=0}} \klabig{k_1+1+k_2+1-k_3-1-k_4-1} u_{j_1}^{k_1} u_{j_2}^{k_2} \overline{u_{j_3}^{k_3}} \overline{u_{j_4}^{k_4}} = 0,
\end{align*}
i.e., $\Ec$ is also conserved along solutions of the differential equation with~$\Bbf$.
\end{proof}

The previous key lemma shows that the right-hand side of the system~\eqref{eq-modsystem} for the modulation functions is a composition of flows of differential equations, in the same way as the numerical method~\eqref{eq-method}, and that these differential equations have the same conserved quantity~$\Ec$. This is the main ingredient to show in the following lemma that
\begin{equation}\label{eq-Ec}
\Ec(t) := \Ec\klabig{\zbf(\ep t)} = \sum_{j\in\disc} \sum_{k\in\Z} (k+1) \absbig{z_j^k(\ep t)}^2,
\end{equation}
where~$\zbf$ are the modulation functions constructed in Section~\ref{sec-mfe} with Algorithm~\ref{algo-comp}, is an almost-invariant of the modulated Fourier expansion. 

\begin{mylemma}[$\Ec$ is an almost-invariant]\label{lemma-invariant}
We have
\[
\absbig{ \Ec(t_{n+1}) - \Ec(t_n) } \lesssim \ep^{N+2} h \myfor 0\le t_n=n h\le \ep^{-1}.
\]
\end{mylemma}
\begin{proof}
  By Lemma~\ref{lemma-flows}, we have
  \[
\Ec(t_n) = \Ec\klabig{\zbf(\ep t_n)} = \Ec\klabig{\Phbf_\Abf^{a_1h}\circ\Phbf_\Bbf^{b_1 h}\circ\dotsm\circ\Phbf_\Abf^{a_sh}\circ\Phbf_\Bbf^{b_s h}\klabig{\zbf(\ep t_n)}} .
  \]
  Since the modulation functions~$z_j^k$ constructed in Section~\ref{sec-mfe} satisfy~\eqref{eq-modsystem-defect}, we thus get 
  \[
\Ec(t_n) = \sum_{j\in\disc} \sum_{k\in\Z} (k+1) \absbig{\e^{-\iu k h} z_j^k\klabig{\ep(t_n+h)}+d_j^k(\ep t_n)}^2.
\]
Rewriting the right-hand side yields
\[
\Ec(t_n) = \Ec(t_{n+1}) + \sum_{j\in\disc} \sum_{k\in\Z} (k+1) \Bigl( \absbig{d_j^k(\ep t_n)}^2 + 2\ReT\klabig{\e^{\iu k h} \overline{z_j^k(\ep t_{n+1})} d_j^k(\ep t_n)} \Bigr).
\]
To estimate $\Ec(t_{n+1}) - \Ec(t_n)$, we use 
\begin{equation}\label{eq-sklak}
k+1 \lesssim \skla{k} \lesssim \skla{k-j^2} \skla{j^2} \le \skla{k-j^2} \skla{j}^2
\end{equation}
similarly as in the proof of Lemma~\ref{lemma-size}. 
Together with the Cauchy--Schwarz inequality this yields
\begin{align*}
 & \absbig{ \Ec(t_{n+1}) - \Ec(t_n) } \lesssim \klabigg{\sum_{j\in\disc} \sum_{k\in\Z} \skla{k}^2 \absbig{d_j^k(\ep t_{n})}^2}^{1/2} \klabigg{\sum_{j\in\disc} \sum_{k\in\Z} \absbig{d_j^k(\ep t_{n})}^2}^{1/2} \\
  &\qquad\qquad  + \klabigg{ \sum_{j\in\disc} \sum_{k\in\Z} \skla{k-j^2}^2 \skla{j}^2 \absbig{z_j^k(\ep t_{n+1})}^2}^{1/2} \klabigg{ \sum_{j\in\disc} \sum_{k\in\Z} \skla{j}^2 \absbig{d_j^k(\ep t_{n})}^2}^{1/2}.
\end{align*}
Using~\eqref{eq-l1l2}, we get
\[
\absbig{ \Ec(t_{n+1}) - \Ec(t_n) } \lesssim \normv{\Kbf\dbf(\ep t_n)}_0 \, \normv{\dbf(\ep t_n)}_0 + \normv{\Labf\zbf(\ep t_{n+1})}_1 \, \normv{\dbf(\ep t_n)}_1,
\]
and the statement of the lemma thus follows from the expansion~\eqref{eq-mfe-exp} and Lemmas~\ref{lemma-size} and~\ref{lemma-defect}.
\end{proof}

We finally recall from~\cite{Gauckler2017a} that the almost-invariant~$\Ec$ of~\eqref{eq-Ec} is close to the squared $H^1(\mathbb{T})$-norm of the numerical solution.

\begin{mylemma}[$\Ec$ is closed to the squared $H^1(\mathbb{T})$-norm]\label{lemma-Ec-norm}
  We have
  \[
\absbig{\Ec(t_n) - \norm{\psi^n}_1^2} \lesssim \ep^3 \myfor 0\le t_n=n h\le \ep^{-1}.
  \]
\end{mylemma}
\begin{proof}
  The proof is essentially the same as the one of \cite[Lemma~6.3]{Gauckler2017a}. On the one hand, we use that
  \[
\absbigg{ \Ec(t_n) - \sum_{j\in\disc} (j^2+1) \absbig{z_j^{j^2}(\ep t_n)}^2 } \le \sum_{j\in\disc} \sum_{k\in\Z : k\ne j^2} (\abs{k}+1) \absbig{z_j^k(\ep t_n)}^2 \lesssim \ep^6.
\]
This follows from~\eqref{eq-l1l2}, the expansion~\eqref{eq-mfe-exp}, the estimate $\normv{\Kbf^{1/2} \zhat_p(\ep t_n)}_0 \lesssim 1$ of Lemma~\ref{lemma-size} and the property $z_{j,1}^k=z_{j,2}^k=0$ for $k\ne j^2$ (see~\eqref{eq-offdiag-zero}).

On the other hand, we write $\eta^n(x) = \sum_{j\in\disc} z_j^{j^2}(\ep t_n) \e^{-\iu j^2 t_n} \e^{\iu j x}$ and use that
\begin{align*}
\absbigg{ \sum_{j\in\disc} (j^2+1) \absbig{z_j^{j^2}(\ep t_n)}^2 - \norm{\psi^n}_1^2} &= \absbig{ \norm{\eta^n}_1^2 - \norm{\psi^n}_1^2}\\ &\le \norm{\eta^n-\psi^n}_1 \klabig{\norm{\eta^n}_1 + \norm{\psi^n}_1} \lesssim \ep^{N+1} + \ep^4. 
\end{align*}
The last inequality follows from $\norm{\eta^n}_1\lesssim \ep$ by~\eqref{eq-mfe-exp} and Lemma~\ref{lemma-size}, from $\norm{\psi^n}_1 \le 2\ep$ by Lemma~\ref{lemma-size-psitilde} and from
\[
\norm{\eta^n-\psi^n}_1 \le \norm{\widetilde\psi^n-\psi^n}_1 + \normbigg{\sum_{j\in\disc} \sum_{k\ne j^2} z_j^k(\ep t_n) \e^{-\iu k t_n} \e^{\iu j x}}_1 \lesssim \ep^{N} + \ep^3 
\]
by~\eqref{eq-l1l2}, \eqref{eq-mfe}, \eqref{eq-mfe-approx}, \eqref{eq-mfe-exp}, \eqref{eq-offdiag-zero} and Lemmas~\ref{lemma-size} and~\ref{lemma-error-psitilde}.
\end{proof}

\section{Long time intervals}\label{sec-long}

In Section~\ref{sec-mfe} and~\ref{sec-estimates}, we have constructed and estimated a modulated Fourier expansion of the numerical solution. In Section~\ref{sec-invariants}, we have shown that this expansion has an almost-invariant that is close to the squared $H^1$-norm of the numerical solution. This implies near-conservation of the squared $H^1$-norm along the numerical solution, as claimed in Theorem~\ref{thm-main}. However, this statement is only true on a not so long time interval $0\le t_n\le\ep^{-1}$ so far. The goal of this section is to extend it to a long time interval of the form $0\le t_n\le\ep^{-N}$. The principle idea is to put many of the not so long time intervals together.

Without loss of generality, we assume in the following that $\np = \ep^{-1} h^{-1}$ is a natural number, so that~$\psi^\np$ is the numerical solution at the endpoint $\ep^{-1}=t_\np$ of the not so long time interval.

\subsection{Modulated Fourier expansion on a new time interval}\label{subsec-mfe-2ndinterval}

We consider a second time interval of length~$\ep^{-1}$ and on it a second modulated Fourier expansion
\[
\sum_{j\in\disc} \sum_{k\in\Z} \hat z_j^k\kla{\ep t} \e^{-\iu kt} \e^{\iu j x} \myfor \ep^{-1}\le t\le 2\ep^{-1}
\]
with modulation functions $\hat\zbf=(\hat{z}_j^k)_{j\in\disc, k\in\Z}$. The construction of the modulation functions is done in the same way as in Section~\ref{subsec-constr}, but with $\psi^{\np}$ as initial value at time~$t_\np=\ep^{-1}$ (instead of~$\psi^0$ at time~$0$). This means that the systems~\eqref{eq-modsystem}, \eqref{eq-modsystem-exp} and \eqref{eq-modsystem-exp-comput} remain unchanged (but now with hats on the variables), and that we replace~\eqref{eq-modsystem-init} by
\[
  \sum_{k\in\mathbb{Z}} \hat z_j^k(1) \e^{-\iu k t_\np} = \psi_j^\np, 
\]
and hence~\eqref{eq-modsystem-exp-init} by
\begin{equation}\label{eq-modsystem-exp-init-2}
\sum_{k\in\mathbb{Z}} \hat{z}_{j,p}^k(1) \e^{-\iu k t_\np} = \begin{cases} \ep^{-1} \psi_j^\np, & p=1,\\ 0, & p\ge 2. \end{cases} 
\end{equation}

We cannot expect the analog $\norm{\psi^\np}_1\le \ep$ of~\eqref{eq-init} to hold for the new initial value, but, as we will see, we can expect
\begin{equation}\label{eq-init-np}
\normbig{\psi^\np}_1 \le 2\ep.
\end{equation}
Under this condition, all estimates of Section~\ref{sec-estimates} for the modulation functions~$\zbf$, the modulation coefficient functions~$\zhat$ and the defect~$\dbf$ (for $0\le t\le\ep^{-1}$) carry over to~$\hat\zbf$, $\hat\zhat$ and~$\hat\dbf$ (for $\ep^{-1}\le t\le2\ep^{-1}$), with possibly different constants.

What we still need in order to patch the intervals $0\le t\le\ep^{-1}$ and $\ep^{-1}\le t\le2\ep^{-1}$ together is a control on the difference $z_j^k(\ep t_\np) - \hat z_j^k(\ep t_\np)$ at the interface $t_\np=\ep^{-1}$ of the two intervals. To estimate this difference, we first derive a Lipschitz-type estimate for the nonlinearity~$\Fhat$ in~\eqref{eq-modsystem-exp-comput} given by~\eqref{eq-F}. We recall that the seminorm~$\abs{\cdot}_{1,q}$ is defined in~\eqref{eq-seminorm}, the norm~$\normv{\cdot}_1$ in~\eqref{eq-normv}, and that $(\Fhat(\vhat))_q=((\Fhat(\vhat))_{j,q}^k)_{j\in\disc, k\in\Z}$.

\begin{mylemma}\label{lemma-F-lipschitz}
  Let $\vhat=\vhat(\tau)\in\mathbb{C}^{\disc\times\Z\times\mathbb{N}}$ and $\hat\vhat=\hat\vhat(\tau)\in\mathbb{C}^{\disc\times\Z\times\mathbb{N}}$ depend on time~$\tau$ with
  \[
    \absBig{\sfrac{\drm^\ell}{\drm\tau^\ell}\vhat(\tau)}_{1,p-2} + \absBig{\sfrac{\drm^\ell}{\drm\tau^\ell}\hat\vhat(\tau)}_{1,p-2}\le M \myfor \ell=0,1,\dots,L.
  \]
  Then, we have
  \begin{align*}
    &\normvbigg{\sum_{q=1}^p \ep^q \sfrac{\drm^\ell}{\drm\tau^\ell} \klaBig{ \klabig{\Fhat(\vhat(\tau))}_q - \klabig{\Fhat(\hat\vhat(\tau))}_q } }_1\\
    &\qquad \le C h \sum_{p'=2}^{p-1} \ep^{p'} \max_{\ell'=0,\dots,\ell} \normvbigg{ \sum_{q=1}^{p-p'} \ep^q \sfrac{\drm^{\ell'}}{\drm\tau^{\ell'}} \klabig{\vhat_q(\tau)-\hat \vhat_q(\tau)}}_1 \qquad\text{for}\quad \ell=0,1,\dots,L
  \end{align*}
  with~$C$ depending only on~$p$, the bound~$M$, the number of derivatives~$L$, the number of stages~$s$ and an upper bound on the absolute values of the coefficients $b_1,\dots,b_s$ of the method~\eqref{eq-method}.
\end{mylemma}
\begin{proof}
  The proof is similar to the one of Lemma~\ref{lemma-F-bounds}, where bounds on~$\Fhat$ have been derived.
We omit the argument~$\tau$ in the following and restrict to the case $L=0$. The extension to larger~$L$ is done as described in part~(b) of the proof of Lemma~\ref{lemma-F-bounds}. 

As~$\Fhat$ of~\eqref{eq-F} is given by compositions of~$\Phihat^{\nt}_{\Ahat}$, $\Phihat^{\nt}_{\Bhat}$ and~$\Phihat^{\nt}_{\Bhat}-\idhat$, we have to prove Lipschitz-type estimates for these operators. For~$\Phihat^{\nt}_{\Ahat}$ we simply have
\[
\normvbigg{\sum_{q=1}^p \ep^q \klaBig{ \klabig{\Phihat^{\nt}_{\Ahat}(\vhat)}_q - \klabig{\Phihat^{\nt}_{\Ahat}(\hat\vhat) }_q } }_1 = \normvbigg{\sum_{q=1}^p \ep^q \kla{ \vhat_q - \hat\vhat_q } }_1.
\]
For~$\Phihat^{\nt}_{\Bhat}$ we have to work harder.
We start from
\begin{align*}
  &\sum_{q=1}^p \ep^q \klaBig{ \klabig{\Phihat^{\nt}_{\Bhat}(\vhat)}_q - \klabig{\Phihat^{\nt}_{\Bhat}(\hat\vhat) }_q} = \sum_{m=0}^\infty \frac{(\iu \nt)^m}{m!} \sum_{p_1+\dots+p_{2m+1}\le p} \ep^{p_1+\dots+p_{2m+1}}\\
  &\qquad \cdot \klaBig{ \kla{\vind_{p_1}-\hat\vind_{p_1}} \ast \vind_{p_2} \ast \dotsm \ast \overline{\vind_{p_{2m+1}}} + \hat\vind_{p_1} \ast \kla{\vind_{p_2}-\hat\vind_{p_2}} \ast \vind_{p_3} \ast \dotsm \ast \overline{\vind_{p_{2m+1}}}\\
    &\qquad\qquad\qquad\qquad\qquad\qquad\qquad + \dots + \hat\vind_{p_1} \ast \dotsm \ast \overline{\hat\vind_{p_{2m}}} \ast \klabig{\overline{\vind_{p_{2m+1}}}-\overline{\hat\vind_{p_{2m+1}}}} }.
\end{align*}
On the right-hand side of this equation we rewrite, for example, 
  \begin{align*}
    &\sum_{p_1+\dots+p_{2m+1}\le p} \ep^{p_1+\dots+p_{2m+1}} \kla{\vind_{p_1}-\hat\vind_{p_1}} \ast \vind_{p_2} \ast \dotsm \ast \overline{\vind_{p_{2m+1}}} \\
    &\qquad = \sum_{p'=2m}^{p-1} \ep^{p'} \klabigg{ \sum_{p_2+\dots+p_{2m+1}= p'}\vind_{p_2} \ast \dotsm \ast \overline{\vind_{p_{2m+1}}}} \ast \klabigg{ \sum_{p_{1}=1}^{p-p'} \ep^{p_{1}} \kla{\vind_{p_{1}}-\hat\vind_{p_{1}}}}.
  \end{align*}
  Using the property~\eqref{eq-algebra} of the discrete convolution then yields the following Lip\-schitz-type estimate for~$\Phihat^{\nt}_{\Bhat}$:
  \begin{align*}
    &\normvbigg{\sum_{q=1}^p \ep^q \klaBig{ \klabig{\Phihat^{\nt}_{\Bhat}(\vhat)}_q - \klabig{\Phihat^{\nt}_{\Bhat}(\hat\vhat) }_q}}_1\\
    &\qquad\qquad \le \sum_{m=0}^\infty \frac{\abs{\nt}^m}{m!} \, C_1^{2m} (2m+1) \sum_{p'=2m}^{p-1} \ep^{p'} \gamma_{p'-1}^{2m} \, \normvbigg{\sum_{p_{1}=1}^{p-p'} \ep^{p_{1}} \kla{\vhat_{p_{1}} - \hat\vhat_{p_{1}}}}_1 \\
    &\qquad\qquad \le  \sum_{p'=0}^{p-1} \ep^{p'} \e^{3 C_1^2 \abs{\nt} \, \gamma_{p'-1}^{2}} \, \normvbigg{\sum_{q=1}^{p-p'} \ep^{q} \kla{\vhat_{q}-\hat\vhat_{q}}}_1,
  \end{align*}
  where $\gamma_{q} = \max(\abs{\vhat}_{1,q},\abs{\hat\vhat}_{1,q})$ for $q\ge 1$, $\gamma_q=1$ for $q\le 0$, and where~$C_1$ is the constant from~\eqref{eq-algebra}.
For~$\Phihat^{\nt}_{\Bhat}-\idhat$ we proceed in the same way but use that the sum over~$m$ starts with $m=1$. This yields
  \begin{align*}
    &\normvbigg{\sum_{q=1}^p \ep^q \klaBig{ \klabig{\klabig{\Phihat^{\nt}_{\Bhat}-\idhat}(\vhat)}_q - \klabig{\klabig{\Phihat^{\nt}_{\Bhat}-\idhat}(\hat\vhat) }_q}}_1\\
    &\qquad\qquad\qquad\qquad \le \sum_{p'=2}^{p-1} \ep^{p'} \klaBig{\e^{3 C_1^2 \abs{\nt} \, \gamma_{p'-1}^{2}} - 1} \normvbigg{\sum_{q=1}^{p-p'} \ep^{q} \kla{\vhat_{q}-\hat\vhat_{q}}}_1.
  \end{align*}
  Putting the estimates of the individual factor of~$\Fhat$ together and using Lemma~\ref{lemma-F-bounds} yields the claimed bound for $L=0$, and the bound for the case $L>0$ is obtained similarly.
\end{proof}

\begin{mylemma}\label{lemma-interface-z}
  In the situation of Sections~\ref{sec-mfe} and~\ref{sec-estimates} and under condition~\eqref{eq-init-np}, we have
  \[
\normvbig{\zbf(1) - \hat\zbf(1)}_1 \lesssim \ep^N.
  \]
\end{mylemma}
\begin{proof}
  We introduce $\ghat=(g_{j,p}^k)_{j\in\disc,k\in\Z,p\in\N}$ by using the notation~\eqref{eq-zhatp} as 
  \[
\ghat_p = \sum_{q=1}^p \ep^q \klabig {\zhat_q - \hat\zhat_q} \myfor p=1,\dots,N
\]
and set $\ghat_p=0$ for $p>N$.
Note that $\ghat_N(1) = \zbf(1) - \hat\zbf(1)$ by~\eqref{eq-mfe-exp}.

(a) The following equations hold for $\ghat_p$ with $p\le N$. From~\eqref{eq-modsystem-exp-comput} we get the equation
\[
  \klabig{\e^{-\iu k h} - \e^{-\iu j^2 h}} g_{j,p}^k(1) + \e^{-\iu k h} \sum_{\ell=1}^{p-1} \, \frac{\ep^\ell h^{\ell}}{\ell!} \frac{\drm^\ell}{\drm\tau^\ell} g_{j,p-\ell}^k(1)  = \sum_{q=1}^p \ep^q \klabig{\Fhat\kla{\zhat(1)} - \Fhat\kla{\hat\zhat(1)}}_{j,q}^k
\]
for~$\ghat$, and from~\eqref{eq-modsystem-exp-init-2} we get
\[
\sum_{k\in\mathbb{Z}} g_{j,p}^k(1) \e^{-\iu k t_\np} = \bigl[\widetilde\psi_j^\np \bigr]^p - \psi_j^\np,
\]
where
\[
\bigl[\widetilde\psi_j^\np\bigr]^p = \sum_{k\in\Z} \sum_{q=1}^p \ep^q z_{j,q}^k(1) \e^{-\iu k t_\np}
\]
is the modulated Fourier expansion as in~\eqref{eq-mfe}, \eqref{eq-mfe-approx} and~\eqref{eq-mfe-exp} but truncated after~$p$ instead of~$N$ terms.

(b) With the same arguments as in the proof of Lemma~\ref{lemma-error-psitilde} on the error $\widetilde\psi^n-\psi^n$ we get (for $n=\np$)
\begin{equation}\label{eq-approx-q}
  \normbig{ \bigl[\widetilde\psi^\np \bigr]^p - \psi^\np}_1 \lesssim \ep^p,
\end{equation}
if we replace the truncation index~$N$ by~$p$ in the proof of Lemma~\ref{lemma-error-psitilde}.

(c) In the same way as in the proof of Lemma~\ref{lemma-size} on the size of~$\zhat_p$, we can now show by induction on $p=1,\dots,N$ that
\[
  \normvBig{\Labf \sfrac{\drm^\ell}{\drm\tau^\ell} \ghat_p(1)}_{1} \lesssim \ep^p \myfor \ell\in\N_0,
\]
if we replace~\eqref{eq-init} by~\eqref{eq-approx-q} and the use of Lemma~\ref{lemma-F-bounds} by Lemma~\ref{lemma-F-lipschitz} in the proof of Lemma~\ref{lemma-size}. 
For $p=N$ and $\ell=0$, this yields the claimed estimate.
\end{proof}

\subsection{The almost-invariant on a new time interval}\label{subsec-invariant-next}

For the new modulated Fourier expansion on $\ep^{-1}\le t\le 2\ep^{-1}$ as considered in the previous Section~\ref{subsec-mfe-2ndinterval} (with modulation functions~$\hat z_j^k$ collected in~$\hat\zbf$), we get a new almost-invariant
\[
\hat \Ec(t) : = \Ec\klabig{\hat\zbf(\ep t)} = \sum_{j\in\disc} \sum_{k\in\Z} (k+1) \absbig{\hat z_j^k(\ep t)}^2
\]
by repeating the arguments of Section~\ref{sec-invariants}. The statements of Section~\ref{sec-invariants} transfer to this new almost-invariant under the assumption~\eqref{eq-init-np} on the size of~$\psi^\np$. In particular, $\hat\Ec$ is nearly preserved on $\ep^{-1}\le t\le 2\ep^{-1}$ and close to the squared $H^1(\mathbb{T})$-norm of the numerical solution. 

What we still need is a control on the difference $\Ec(\ep^{-1}) - \hat \Ec(\ep^{-1})$ at the interface of the two time intervals.

\begin{mylemma}\label{lemma-interface-E}
  In the situation of Sections~\ref{sec-mfe} and~\ref{sec-estimates} and under condition~\eqref{eq-init-np}, we have
  \[
\absbig{\Ec(\ep^{-1}) - \hat \Ec(\ep^{-1})} \lesssim \ep^{N+1}.
  \]
\end{mylemma}
\begin{proof}
  Omitting the argument~$1$ of~$z_j^k$ and~$\hat z_j^k$ we have
  \[
\absbig{\Ec(\ep^{-1}) - \hat \Ec(\ep^{-1})} \le \sum_{j\in\disc} \sum_{k\in\Z} (\abs{k}+1) \absbig{z_j^k-\hat z_j^k} \cdot \klabig{\abs{z_j^k}+\abs{\hat z_j^k} }.
  \]
  With the estimate~\eqref{eq-sklak} of $\abs{k}+1$, the Cauchy--Schwarz inequality and~\eqref{eq-l1l2} we thus get
  \[
\absbig{\Ec(\ep^{-1}) - \hat \Ec(\ep^{-1})} \le \normv{\zbf-\hat\zbf}_1 \klabig{ \normv{\Labf \zbf}_1 + \normv{\Labf\hat\zbf}_1}.
\]
Lemmas~\ref{lemma-size} and~\ref{lemma-interface-z} together with the expansion~\eqref{eq-mfe-exp} finally yield the statement of the lemma. 
\end{proof}

\subsection{Proof of Theorem~\ref{thm-main}}

The proof of Theorem~\ref{thm-main} is obtained by patching many time intervals of length~$\ep^{-1}$ together to a long time interval of length~$\ep^{-N}$ as in \cite[Section~7.2]{Gauckler2017a}. On every time interval of the form $\ell \ep^{-1}\le t \le (\ell+1)\ep^{-1}$ we construct a modulated Fourier expansion of the numerical solution and get a corresponding almost-invariant that we denote by $\Ec_\ell(t)$ (in particular, we have $\Ec_0=\Ec$ and $\Ec_1=\hat\Ec$ with~$\Ec$ from Section~\ref{sec-invariants} and~$\hat\Ec$ from Section~\ref{subsec-invariant-next}). By Lemmas~\ref{lemma-invariant} and~\ref{lemma-interface-E} we have for all~$\ell$
\[
\absbig{ \Ec_\ell(t_n) - \Ec_0(t_0) } \le C (\ell+1) \ep^{N+1} \myfor \ell \ep^{-1}\le t_n=n h\le (\ell+1)\ep^{-1}
\]
provided that $\norm{\psi^0}_1\le\ep$ and $\norm{\psi^{m}}_1\le 2\ep$ for $t_m=\ep^{-1},2\ep^{-1},\ldots,\ell\ep^{-1}$. For $\ell+1\le \ep^{-N+2}$ (and hence $t_n\le \ep^{-N+1}$), this implies by Lemma~\ref{lemma-Ec-norm}
\[
\absbig{ \norm{\psi^n}_1^2 - \norm{\psi^0}_1^2} \le C \ep^3 \myfor \ell \ep^{-1}\le t_n=n h\le (\ell+1)\ep^{-1}.
\]
Besides proving near-conservation of the squared $H^1(\mathbb{T})$-norm as stated in Theorem~\ref{thm-main}, this also shows that the hypothesis $\norm{\psi^{m}}_1\le 2\ep$ also holds for $t_m=(\ell+1)\ep^{-1}$ provided that~$\ep$ is sufficiently small, and hence concludes the proof of Theorem~\ref{thm-main}.

\section{Conclusion}\label{sec-conclusion}

We have proven long-time near-conservation of energy and the squared $H^1$-norm for splitting integrators when applied to the cubic nonlinear Schr\"odinger equation in a weakly nonlinear regime. The main novelty is that the analysis includes all consistent splitting integrator with real-valued coefficients, in particular higher order splitting integrators.

As in previous papers, the given proof of long-time near-conservation of energy is based on a modulated Fourier expansion. The main difference is that higher order splitting integrators involve multiple nonlinear steps, which complicates the construction of such an expansion and in particular the derivation of suitable estimates in Sections~\ref{sec-mfe} and~\ref{sec-estimates}. A crucial point is to set up the system for the coefficients of the modulated Fourier expansion in such a way that the splitting structure can be immediately read off (see Equations~\eqref{eq-modsystem} and~\eqref{eq-modsystem-exp}), which is of particular importance for the derivation of the almost-invariant in Section~\ref{sec-invariants}. 

Focusing on these difficulties coming from higher order splitting integrators, we did not optimize the results with respect to the CFL-type step-size restriction~\eqref{eq-cfl}, but it is clear that a combination of the techniques presented here with those of~\cite{Gauckler2017a} lead to a weaker assumption in the form of a numerical non-resonance condition. In addition, the presented way to deal with higher order splitting integrators should also allow us to extend the results of~\cite{Faou2009a,Faou2009b,Gauckler2010a,Gauckler2010b} on long-time near-conservation of actions and energy in non-resonant situations and the results of~\cite{Faou2014} on long-time orbital stability of plane wave solutions to higher order splittings. We also expect that the techniques can be used to study higher order splitting integrators for semilinear wave equations, for which splitting integrators coincide with symplectic extended Runge--Kutta--Nystr\"om methods~\cite{Wu2013} as has been shown in~\cite{Blanes2015}.

\subsection*{Acknowledgement}

This work was supported by Deutsche Forschungsgemeinschaft through SFB 1114.

\end{document}